\documentclass[10pt,a4paper,reqno]{amsart}
\usepackage{amsmath,amssymb,amsthm,amsfonts}
\usepackage[export]{adjustbox}
\usepackage{amsbsy}
\usepackage[utf8]{inputenc}
\usepackage[normalem]{ulem}
\usepackage{xspace}
\usepackage{cancel}
\usepackage[greek,english]{babel}
\usepackage{url}
\usepackage{wrapfig}
\usepackage{enumitem}
\usepackage{multicol}
\usepackage{moreenum}
\usepackage{xcolor}
\usepackage{subcaption}
\usepackage{comment}
\usepackage{tikz}
\usetikzlibrary{backgrounds}
\usepackage[bottom=2.5cm,top=2.5cm, left=2.5cm, right=2.5cm]{geometry}

\usepackage{multirow}
\usepackage{array,multirow}

\definecolor{LinkColor}{rgb}{0,0,0} 
\usepackage[colorlinks=true,linkcolor=LinkColor,citecolor=LinkColor,urlcolor= LinkColor, naturalnames, hyperindex, pdfstartview=FitH, bookmarksnumbered, plainpages]{hyperref}
\usepackage{cleveref}

\makeatletter
\newcommand{\slunlhd}{%
  \mathrel{\mathpalette\sl@unlhd\relax}%
}

\newcommand{\sl@unlhd}[2]{%
  \sbox\z@{$#1\lhd$}%
  \sbox\tw@{$#1\leqslant$}%
  \dimen@=\ht\tw@
  \advance\dimen@-\ht\z@
  \ifx#1\displaystyle
    \advance\dimen@ .2pt
  \else
    \ifx#1\textstyle
      \advance\dimen@ .2pt
    \fi
  \fi
  \ooalign{\raisebox{\dimen@}{$\m@th#1\lhd$}\cr$\m@th#1\leqslant$\cr}%
}
\makeatother

\newtheorem{maintheorem}{Theorem}[section]
\newtheorem{maincorollary}[maintheorem]{Corollary}
\newtheorem{mainquestion}[maintheorem]{Question}

\newtheorem{theorem}{Theorem}[section]
\newtheorem{corollary}[theorem]{Corollary}
\newtheorem{lemma}[theorem]{Lemma}
\newtheorem{proposition}[theorem]{Proposition}

\theoremstyle{definition}

\newtheorem{remark}[theorem]{Remark}

\newcommand{\cut}{\textsf{cut}\xspace}
\newcommand{\SG}[1]{\text{SG}[#1]}
\newcommand{\SL}{\operatorname{SL}}
\newcommand{\GL}{\operatorname{GL}}
\newcommand{\PSL}{\operatorname{PSL}}
\newcommand{\PGL}{\operatorname{PGL}}

\newcommand{\Aut}{\operatorname{Aut}}
\newcommand{\Out}{\operatorname{Out}}
\newcommand{\Inn}{\operatorname{Inn}}

\newcommand{\FF}{\mathbb{F}}
\newcommand{\Z}{\textup{Z}}

\newcommand{\V}{\textup{V}}
\newcommand{\C}{\textup{C}}
\newcommand{\N}{\textup{N}}
\newcommand{\ZZ}{\mathbb{Z}}
\newcommand{\QQ}{\mathbb{Q}}
\newcommand{\GK}{\Gamma_{\textup{GK}}}
\renewcommand{\O}{\operatorname{O}}
\newcommand{\F}{\textup{F}}
\newcommand{\Frob}{\textup{Fr}} 
 
 \newcommand{\PQ}{(\textup{PQ})\xspace}

\newcommand{\GEN}[1]{\langle #1 \rangle}
\newcommand{\pmatriz}[1]{\left(\begin{array} #1 \end{array}\right)}
\newcommand{\qand}{\quad \text{and} \quad}

\DeclareMathOperator{\lcm}{lcm}
\title{Gruenberg-Kegel graphs: cut groups, rational groups and the Prime Graph Question}

\author[A.~B\"achle]{Andreas B\"achle}
\address{(Andreas Bächle)}
\email{\href{mailto:ABaechle@gmx.net}{ABaechle@gmx.net}}

\author[A. Kiefer]{Ann Kiefer}
\address{(Ann Kiefer) LUCET, Université du Luxembourg, Maison des Sciences Humaines, 11, Porte des Sciences, L-4366 Esch-sur-Alzette}
\email{\href{mailto:ann.kiefer@uni.lu}{ann.kiefer@uni.lu}}

\author[S. Maheshwary]{Sugandha Maheshwary}
\address{(Sugandha Maheshwary) Indian Institute of Science Education and Research, Mohali, Sector 81, Mohali (Punjab)-140306, India.}
\email{\href{mailto:sugandha@iisermohali.ac.in}{sugandha@iisermohali.ac.in}}

\author[A. del R\'{\i}o]{\'{A}ngel del R\'{\i}o}
\address{(\'{A}ngel del R\'{\i}o) Departamento de Matem\'{a}ticas, Universidad de Murcia, 30100, Murcia, Spain.}
\email{\href{mailto:sugandha@iisermohali.ac.in}{adelrio@um.es}}

\thanks{The work of the first author was partially supported by a postdoctoral fellowship of the FWO (Research Foundation Flanders). The second author is grateful to Onderzoeksraad Vrije Universiteit Brussel and to the Luxembourg Centre for Educational Testing. The third author gratefully acknowledges the support by DST (Department of Science and Technology), India (INSPIRE/04/2017/000897) and the local hospitality provided by \'{A}. del R\'{i}o for stay at Universidad de Murcia, Murcia, Spain which played a vital role in the outcome of this paper. The last author is partially supported by Grant 19880/GERM/15 funded by Fundaci\'{o}n S\'{e}neca of Murcia, and Grant PID2020-113206GB-I00 funded by MCIN/AEI/10.13039/501100011033.}


\begin{document}

\maketitle

\begin{abstract} The Gruenberg-Kegel graph of a group is the undirected graph whose vertices are those primes which occur as the order of an element of the group, and distinct vertices $p$, $q$ are  joined by an edge whenever the group has an element of order $pq$. 
	It reflects interesting properties of the group.
	A group is said to be \cut if the central units of its integral group ring are trivial. 
	This is a rich family of groups, which contains the well studied class of rational groups, and has received attention recently. In the first part of this paper we give a complete classification of the Gruenberg-Kegel graphs of finite solvable \cut groups which have at most three elements in their prime spectrum. For the remaining cases of finite solvable \cut groups, we strongly restrict the list of the possible Gruenberg-Kegel graphs and realize most of them by finite solvable \cut groups. Likewise, we give a list of the possible Gruenberg-Kegel graphs of finite solvable rational groups and realize  as such all but one of them. As an application, we completely classify the Gruenberg-Kegel graphs of metacyclic, metabelian, supersolvable, metanilpotent and $2$-Frobenius groups  for the classes of \cut  groups and rational groups, respectively. The Prime Graph Question asks whether the Gruenberg-Kegel graph of a group coincides with that of the group of normalized units of its integral group ring. The recent appearance of a counter-example for the First Zassenhaus Conjecture on the torsion units of integral group rings has highlighted the relevance of this question. We answer the Prime Graph Question for integral group rings for finite rational groups and most finite \cut groups.
\end{abstract}

\section{Introduction}

Let $G$ be a group. 
The \emph{Gruenberg-Kegel graph of $G$}, which we abbreviate here as \emph{GK-graph of $G$} and denote by $\GK(G)$, is the undirected loop-free and multiple-free graph whose vertices are the primes which occur as orders of elements of $G$, and two vertices $p$ and $q$ are joined by an edge, if and only if $G$ contains an element of order $pq$. 
In many references the GK-graph of $G$ is also called the \emph{prime graph} of $G$. 
This naively built graph encodes interesting properties of a group. For instance, the GK-graph of a finite solvable group is disconnected if, and only if, the group is Frobenius or $2$-Frobenius (see \Cref{Frobenius} for definition and references). Also for arbitrary finite groups it reflects the structure of Hall subgroups (see, for example \cite[Theorem~3]{Williams}). Several groups are even completely determined by their GK-graph, e.g.\ most $\PSL(2,p)$ and all $\PGL(2,p^k)$ for $p$ a prime and $k \geqslant 2$ \cite{KKK07,AKK10} to name a few popular ones. For these reasons, the study of the GK-graph is an active field of research \cite{AKK10,ZM13,GKKM14,BC15,GKLNS,MP16,BS19,CM,GV21}. 

If $\Gamma$ is the GK-graph of a group of a given type then we say that $\Gamma$ is \emph{realizable as the GK-graph} of groups of that type. We assume all graphs to be undirected loop-free and multiple-free and all vertices are labeled by distinct prime numbers. Every graph $\Gamma$ is realizable as the GK-graph of a group, namely the free product $*_{n\in X} C_n$, where $X$ is formed by the vertices of $\Gamma$ and the products $pq$ with $p-q$ an edge of $\Gamma$. However, for finite groups, the situation is much more restrictive. 
For example, the number of connected components of a graph realizable as the GK-graph of a finite group is at most 6 \cite{Williams,Kon90}. 
Recently there was interest in deciding which graphs are realizable as the GK-graph of a finite group and a complete answer was, for instance, obtained for complete bipartite graphs in \cite{MP16}. Moreover the graphs realizable as GK-graphs of finite solvable groups were explicitly characterized in \cite{GKLNS}.

A class of groups that is of recent interest is that of \cut groups, where  a group is said to be \cut if all \textbf{c}entral \textbf{u}nits in its integral group ring are \textbf{t}rivial, i.e.,\ scalar multiples of central group elements \cite{BMP17,Mah18,Bac,BCJM,Tre19,Bac19,BBM20,Gri20,Mor22}. 
The term \cut was coined in \cite{BMP17}. 
It was observed later that, by a previous theorem 
(see \cite{RS} or \cite[Corollary~7.1.15]{JdR1}), the finite \cut groups coincide with those studied by group theorists, under the name of inverse semi-rational groups \cite{CD}.
The class of \cut groups contains the class of finite rational groups but is considerably bigger, e.g.\ $86.62 \%$ of all groups up to order $512$ are \cut while only $0.57\%$ of them are rational \cite[Section~7]{BCJM}. Since it turned out that $G$ being a \cut group is a major obstruction for certain fixed point properties (such as Kazhdan's property (T)) of the unit group of its integral group ring $\mathbb{Z}G$, this class of groups also appeared naturally in the study of these properties and in the proof of a virtual unit theorem for non-trivial amalgamated products \cite{BJJKT18_1,BJJKT18_2}.

The main goal of this article is to contribute to the classification of the GK-graphs of finite solvable \cut groups. In the process, we also study the GK-graphs of finite rational groups.
Relevant information on the elements of prime power order in such groups, and hence on the vertices of their GK-graphs, is well known by results in \cite{Gow,CD,Bac,Mah18}. 
For example,  a finite solvable \cut (respectively, rational) group has at most $4$ (respectively, $3$) vertices (see \Cref{2357} below). Most of the work in the present paper concerns the edges of the GK-graph of a solvable \cut group. The GK-graphs of finite solvable \cut groups with at most 3 vertices are completely classified in our first result.

\begin{maintheorem}\label{TLess4}
A graph with at most three vertices is realizable as the GK-graph of a non-trivial finite solvable \cut group if and only if it is one of those appearing in \Cref{Less4}. 

\begin{figure}[h]

	\begin{tabular}{|c|c|}
		\hline
		  & Graphs \\\hline
		1 vertex & \emph{(a)}\hspace{-0.8cm}
		\begin{subfigure}{.15\textwidth}
			\centering 
			\begin{tikzpicture};
			\node[label=west:{$2$}] at (0,0) (2){};
			\foreach \p in {2}{
				\draw[fill=black] (\p) circle (0.075cm);
			}
			\end{tikzpicture}
		\end{subfigure}
		\hspace{3cm} 
		\emph{(b)}\hspace{-0.8cm}
		\begin{subfigure}{.15\textwidth}
			\centering 
			\begin{tikzpicture}
			\node[label=east:{$3$}] at (0.5,1) (3){};
			\foreach \p in {3}{
				\draw[fill=black] (\p) circle (0.075cm);
			}
			\end{tikzpicture}
		\end{subfigure} 
		\\\hline
		2 vertices & \emph{(c)}
		\hspace{-0.5cm}
		\begin{subfigure}{.15\textwidth}
			\centering 
			\begin{tikzpicture}
			\node[label=west:{$2$}] at (0,1) (2){};
			\node[label=east:{$3$}] at (0.5,1) (3){};
			\foreach \p in {2,3}{
				\draw[fill=black] (\p) circle (0.075cm);
			}
			\end{tikzpicture}
		\end{subfigure}
		\emph{(d)}
		\hspace{-0.5cm}
		\begin{subfigure}{.15\textwidth}
			\centering 
			\begin{tikzpicture}
			\node[label=west:{$2$}] at (0,1) (2){};
			\node[label=east:{$3$}] at (0.5,1) (3){};
			\foreach \p in {2,3}{
				\draw[fill=black] (\p) circle (0.075cm);
			}
			\draw (2)--(3);
			\end{tikzpicture}
		\end{subfigure} 
		\emph{(e)}
		\hspace{-1cm}
		\begin{subfigure}{.15\textwidth}
			\centering
			\begin{tikzpicture}
			\node[label=west:{$2$}] at (0,0.5) (2){};
			\node[label=west:{$5$}] at (0,0) (3){};
			\foreach \p in {2,3}{
				\draw[fill=black] (\p) circle (0.075cm);
			}
			\end{tikzpicture}
		\end{subfigure}
		\emph{(f)}
		\hspace{-1cm}
		\begin{subfigure}{.15\textwidth}
			\centering 
			\begin{tikzpicture}
			\node[label=west:{$2$}] at (0,0.5) (2){};
			\node[label=west:{$5$}] at (0,0) (5){};
			\foreach \p in {2,3}{
				\draw[fill=black] (\p) circle (0.075cm);
			}
			\draw (2)--(3);
			\end{tikzpicture}
		\end{subfigure}
		\emph{(g)}
		\hspace{-1cm}
		\begin{subfigure}{.15\textwidth}
			\centering 
			\begin{tikzpicture}
			\node[label=east:{$3$}] at (0.5,0.5) (3){};
			\node[label=east:{$7$}] at (0.5,0) (7){};
			\foreach \p in {3,7}{
				\draw[fill=black] (\p) circle (0.075cm);
			}
			\end{tikzpicture}
		\end{subfigure}
		\\\hline
		& 
		\emph{(h)}\hspace{-0.3cm}
		\begin{subfigure}{.15\textwidth}
			\centering
			\begin{tikzpicture}
			\node[label=west:{$2$}] at (0,0.5) (2){};
			\node[label=east:{$3$}] at (0.5,0.5) (3){};
			\node[label=west:{$5$}] at (0,0) (5){};
			\foreach \p in {2,3,5}{
				\draw[fill=black] (\p) circle (0.075cm);
			}
			\draw (2)--(3);
			\end{tikzpicture}
		\end{subfigure}
		\emph{(i)}\hspace{-0.3cm}
		\begin{subfigure}{.15\textwidth}
			\centering
			\begin{tikzpicture}
			\node[label=west:{$2$}] at (0,0.5) (2){};
			\node[label=east:{$3$}] at (0.5,0.5) (3){};
			\node[label=west:{$5$}] at (0,0) (5){};
			\foreach \p in {2,3,5}{
				\draw[fill=black] (\p) circle (0.075cm);
			}
			\draw (2)--(3);
			\draw (2)--(5);
			\end{tikzpicture}
		\end{subfigure} 
		\emph{(j)}\hspace{-0.3cm}
		\begin{subfigure}{.15\textwidth}
			\centering
			\begin{tikzpicture}
			\node[label=west:{$2$}] at (0,0.5) (2){};
			\node[label=east:{$3$}] at (0.5,0.5) (3){};
			\node[label=west:{$5$}] at (0,0) (5){};
			\foreach \p in {2,3,5}{
				\draw[fill=black] (\p) circle (0.075cm);
			}
			\draw (2)--(3);
			\draw (3)--(5);
			\end{tikzpicture}
		\end{subfigure}
		\emph{(k)}\hspace{-0.3cm}
		\begin{subfigure}{.15\textwidth}
			\centering
			\begin{tikzpicture}
			\node[label=west:{$2$}] at (0,0.5) (2){};
			\node[label=east:{$3$}] at (0.5,0.5) (3){};
			\node[label=west:{$5$}] at (0,0) (5){};
			\foreach \p in {2,3,5}{
				\draw[fill=black] (\p) circle (0.075cm);
			}
			\draw (2)--(3);
			\draw (2)--(5);
			\draw (3)--(5);
			\end{tikzpicture}
		\end{subfigure}
		\\
		3 vertices& \\
		& \emph{(l)}\hspace{-0.3cm}
		\begin{subfigure}{.15\textwidth}
			\centering 
			\begin{tikzpicture}
			\node[label=west:{$2$}] at (0,0.5) (2){};
			\node[label=east:{$3$}] at (0.5,0.5) (3){};
			\node[label=east:{$7$}] at (0.5,0) (7){};
			\foreach \p in {2,3,7}{
				\draw[fill=black] (\p) circle (0.075cm);
			}
			\draw (2)--(3);
			\end{tikzpicture}
		\end{subfigure}
		\emph{(m)}\hspace{-0.3cm}
		\begin{subfigure}{.15\textwidth}
			\centering
			\begin{tikzpicture}
			\node[label=west:{$2$}] at (0,0.5) (2){};
			\node[label=east:{$3$}] at (0.5,0.5) (3){};
			\node[label=east:{$7$}] at (0.5,0) (7){};
			\foreach \p in {2,3,7}{
				\draw[fill=black] (\p) circle (0.075cm);
			}
			\draw (2)--(3);
			\draw (2)--(7);
			\end{tikzpicture}
		\end{subfigure}
		\emph{(n)}\hspace{-0.3cm}
		\begin{subfigure}{.15\textwidth}
			\centering
			\begin{tikzpicture}
			\node[label=west:{$2$}] at (0,0.5) (2){};
			\node[label=east:{$3$}] at (0.5,0.5) (3){};
			\node[label=east:{$7$}] at (0.5,0) (7){};
			\foreach \p in {2,3,7}{
				\draw[fill=black] (\p) circle (0.075cm);
			}
			\draw (2)--(3);
			\draw (3)--(7);
			\end{tikzpicture}
		\end{subfigure} 
		\emph{(o)}\hspace{-0.3cm}
		\begin{subfigure}{.15\textwidth}
			\centering 
			\begin{tikzpicture}
			\node[label=west:{$2$}] at (0,0.5) (2){};
			\node[label=east:{$3$}] at (0.5,0.5) (3){};
			\node[label=east:{$7$}] at (0.5,0) (7){};
			\foreach \p in {2,3,7}{
				\draw[fill=black] (\p) circle (0.075cm);
			}
			\draw (2)--(3);
			\draw (2)--(7);
			\draw (3)--(7);
			\end{tikzpicture}
		\end{subfigure}
		\\\hline
	\end{tabular}
\caption{\label{Less4} GK-graphs of solvable \cut groups with at most 3 vertices.}
\end{figure}
\end{maintheorem}

\begin{maincorollary}
Let $\Gamma$ be a graph with three vertices. Then $\Gamma$ is the GK-graph of a finite solvable \cut group if and only if $2-3$ is an edge of $\Gamma$ and either $5$ or $7$ is a vertex of $\Gamma$. 
\end{maincorollary}

For the remaining case of graphs with four vertices we obtain the following result which strongly restricts the possible GK-graphs of finite solvable \cut groups.

\begin{maintheorem}\label{T4}
	Each graph in the upper line in \Cref{Four} is realizable as the GK-graph of a finite solvable \cut group. Any other graph with at least four vertices which is realizable as the GK-graph of a finite {non-trivial} solvable \cut group is one of the four in the lower line of \Cref{Four}.
\begin{figure}[h]
	\begin{tabular}{|c|c|c|c|c|}
		\hline
		 & 4 edges & 5 edges & 6 edges \\\hline
		Verified &  
		\begin{subfigure}{0\textwidth}
		\setcounter{subfigure}{15}
		\end{subfigure}
		\emph{(p)}\hspace{-0.4cm}
		\begin{subfigure}{.15\textwidth}
			\centering
			\begin{tikzpicture}
			\node[label=west:{$2$}] at (0,0.5) (2){};
			\node[label=east:{$3$}] at (0.5, 0.5) (3){};
			\node[label=west:{$5$}] at (0, 0) (5){};
			\node[label=east:{$7$}] at (0.5, 0) (7){};
			\foreach \p in {2,3,5,7}{
				\draw[fill=black] (\p) circle (0.075cm);
			}
			\draw (2)--(3);
			\draw (2)--(7);;
			\draw (3)--(5);
			\draw (5)--(7);
			\end{tikzpicture}
		\end{subfigure}
		& \emph{(q)}\hspace{-0.4cm}
		\begin{subfigure}{.15\textwidth}
			\centering
			\begin{tikzpicture}
			\node[label=west:{$2$}] at (0,0.5) (2){};
			\node[label=east:{$3$}] at (0.5, 0.5) (3){};
			\node[label=west:{$5$}] at (0, 0) (5){};
			\node[label=east:{$7$}] at (0.5, 0) (7){};
			\foreach \p in {2,3,5,7}{
				\draw[fill=black] (\p) circle (0.075cm);
			}
			\draw (2)--(3);
			\draw (2)--(5);
			\draw (2)--(7);;
			\draw (3)--(5);
			\draw (5)--(7);
			\end{tikzpicture}
		\end{subfigure}
		& \emph{(r)}\hspace{-0.4cm}
		\begin{subfigure}{.15\textwidth}
			\centering
			\begin{tikzpicture}
			\node[label=west:{$2$}] at (0,0.5) (2){};
			\node[label=east:{$3$}] at (0.5, 0.5) (3){};
			\node[label=west:{$5$}] at (0, 0) (5){};
			\node[label=east:{$7$}] at (0.5, 0) (7){};
			\foreach \p in {2,3,5,7}{
				\draw[fill=black] (\p) circle (0.075cm);
			}
			\draw (2)--(3);
			\draw (2)--(5);
			\draw (2)--(7);;
			\draw (3)--(5);
			\draw (3)--(7);
			\draw (5)--(7);
			\end{tikzpicture}
		\end{subfigure}
		\\\hline
		Possible & \emph{(s)}\hspace{-0.4cm}
		\begin{subfigure}{.15\textwidth}
			\centering
			\begin{tikzpicture}
			\node[label=west:{$2$}] at (0,0.5) (2){};
			\node[label=east:{$3$}] at (0.5, 0.5) (3){};
			\node[label=west:{$5$}] at (0, 0) (5){};
			\node[label=east:{$7$}] at (0.5, 0) (7){};
			\foreach \p in {2,3,5,7}{
				\draw[fill=black] (\p) circle (.075cm);
			}
			\draw (2)--(3);
			\draw (2)--(7);
			\draw (3)--(5);
			\draw (3)--(7);
			\end{tikzpicture}
		\end{subfigure}
		\emph{(t)}\hspace{-0.4cm}
		\begin{subfigure}{.15\textwidth}
			\centering
			\begin{tikzpicture}
			\node[label=west:{$2$}] at (0,0.5) (2){};
			\node[label=east:{$3$}] at (0.5, 0.5) (3){};
			\node[label=west:{$5$}] at (0, 0) (5){};
			\node[label=east:{$7$}] at (0.5, 0) (7){};
			\foreach \p in {2,3,5,7}{
				\draw[fill=black] (\p) circle (.075cm);
			}
			\draw (2)--(3);
			\draw (2)--(5);
			\draw (2)--(7);
			\draw (3)--(5);
			\end{tikzpicture}
		\end{subfigure}
		& \emph{(u)}\hspace{-0.4cm}
        \begin{subfigure}{.15\textwidth}
			\centering
			\begin{tikzpicture}
			\node[label=west:{$2$}] at (0,0.5) (2){};
			\node[label=east:{$3$}] at (0.5, 0.5) (3){};
			\node[label=west:{$5$}] at (0, 0) (5){};
			\node[label=east:{$7$}] at (0.5, 0) (7){};
			\foreach \p in {2,3,5,7}{
				\draw[fill=black] (\p) circle (.075cm);
			}
			\draw (2)--(3);
			\draw (2)--(7);
			\draw (3)--(5);
			\draw (3)--(7);
			\draw (5)--(7);
			\end{tikzpicture}
		\end{subfigure} 
		\emph{(v)}\hspace{-0.4cm}
		\begin{subfigure}{.15\textwidth}
			\centering
			\begin{tikzpicture}
			\node[label=west:{$2$}] at (0,0.5) (2){};
			\node[label=east:{$3$}] at (0.5, 0.5) (3){};
			\node[label=west:{$5$}] at (0, 0) (5){};
			\node[label=east:{$7$}] at (0.5, 0) (7){};
			\foreach \p in {2,3,5,7}{
				\draw[fill=black] (\p) circle (.075cm);
			}
			\draw (2)--(3);
			\draw (2)--(5);
			\draw (2)--(7);
			\draw (3)--(5);
			\draw (3)--(7);
			\end{tikzpicture}
		\end{subfigure}
		&
		\\\hline
	\end{tabular}
	\caption{\label{Four} GK-graphs of finite solvable \cut groups with 4 vertices.}
\end{figure}

\end{maintheorem}

We also deduce which  graphs can be realized as the GK-graph of a finite solvable rational group. 

\begin{maintheorem}\label{PrimeGraphRat}
Let $G$ be a non-trivial finite solvable rational group. 
Then $\GK(G)$ is one of the graphs (a), (c), (d), (e), (f), (i) or (k) in Figure~\ref{Less4}.
Moreover, any of these graphs, except possibly (i), is realizable as the GK-graph of a solvable rational group.
\end{maintheorem}

Therefore, to complete the classification of the GK-graphs of finite solvable \cut (respectively, rational) groups, it remains to answer the following questions.

\begin{mainquestion}\label{QuestionUndecidedCut}
Which of the four graphs in the second line of \Cref{Four} are realizable as the GK-graphs of some finite solvable \cut group?
\end{mainquestion}

\begin{mainquestion}\label{Question325} Is $3-2-5$ the GK-graph of a finite solvable rational group? 
\end{mainquestion}

To the best of our knowledge, the answer to these questions is unknown.\\

As an application of the above theorems we completely classify the graphs realizable as GK-graphs of finite metacyclic, metabelian, supersolvable, metanilpotent and $2$-Frobenius \cut (rational) groups.

\medskip

The GK-graph of a finite group $G$ is also intimately related to the structure of the integral group ring $\mathbb{Z}G$. For example, the augmentation ideal of $\mathbb{Z}G$ is indecomposable as $\mathbb{Z} G$-module if and only if $\GK(G)$ is connected (see \cite{GR75} and \cite[Theorem~6]{Williams}).
Recall that the augmentation of a group ring element is the sum of its coefficients and the augmentation ideal is formed by the elements of augmentation $0$.
If $\V(\ZZ G)$ denotes the group of normalized units (units of augmentation one in $\ZZ G$), then by a classical result of G.~Higman, the  GK-graphs of $G$ and $\V(\ZZ G)$ have the same set of vertices
\cite[Theorem~12]{Higman1940Thesis} (see also \cite{Sandling1981}).
Kimmerle conjectured that, on the level of GK-graphs, $G$ and {the} {unit group of} its integral group ring carry exactly the same information.
Precisely, he asked the following \cite{Kim06}:

	\begin{quote}
\textbf{The Prime Graph Question (PQ)}. For a finite group $G$, do  the GK-graphs of $G$ and $\V(\ZZ G)$ coincide?
	\end{quote}

Kimmerle was led to ask this question as an approximation to the {\emph{First Zassenhaus Conjecture}} (ZC1), which states that every unit of augmentation $1$ in $\ZZ G$ is conjugate to some group element by a unit of $\QQ G$. Though this conjecture has been settled recently in the negative \cite{EiseleMargolis18}, it is still open for \cut groups.
Kimmerle immediately gave an affirmative answer to (PQ) for all finite solvable groups and Frobenius groups \cite{Kim06}. Also for many non-solvable groups, this question has been answered positively in the meantime using a reduction result of Kimmerle and Konovalov to almost simple groups \cite{KimmerleKonovalov2016}. However, in general (PQ) remains open.  
The Prime Graph Question is also a weaker form of one of the main open problems in integral group rings, namely the \emph{{S}pectrum {P}roblem} which asks whether the set of orders of elements in $G$ and of those in $\V(\ZZ G)$ coincide. This is also listed in the collection of problems in Sehgal's book on units in integral group rings \cite[Problem 8]{Sehgal93}. The {S}pectrum {P}roblem is known to have a positive answer e.g. for all solvable groups \cite{Her}.

 We prove the following:

\begin{maintheorem}\label{PrimeGraphCut} Let $G$ be a finite \cut group such that there is no epimorphism $G \to M$, where $M$ denotes the sporadic simple Monster group. Then (PQ) has a positive answer for $G$. 
\end{maintheorem}

As a consequence, we answer (PQ) for rational groups.

\begin{maincorollary}\label{PrimeGraphRational}
	The Prime Graph Question has a positive answer for finite rational groups.
\end{maincorollary}

We now outline the contents of the paper. 
In \Cref{Preliminaries} we introduce the basic notation and collect some general properties about \cut groups. 
It concludes by showing a list of solvable \cut groups realizing all the graphs in \Cref{Less4} or the first line of \Cref{Four} (see \Cref{PGroups}). 
This proves the sufficency direction of Theorems \ref{TLess4}, \ref{T4} and \ref{PrimeGraphRat}. 
The proofs of \Cref{TLess4} and \Cref{PrimeGraphRat} are completed in \Cref{SectionTLess4}, while the proof of \Cref{T4} is completed in \Cref{Section4}. 
\Cref{SectionApplications} is dedicated to classify the GK-graphs realizable by groups in several classes of \cut or rational groups. 
Finally, we prove \Cref{PrimeGraphCut} in \Cref{SectionPQ}.

\section{Notation and Preliminaries}\label{Preliminaries}

\subsection{Basic notation}
The cardinality of a set $X$ is denoted by $|X|$ and we use the standard notation $\varphi$, $\gcd$ and $\lcm$ for Euler's totient function, greatest common divisor and least common multiple, respectively. As it is customary, often we are implicitly assuming that $p$ denotes a prime number and hence notation as for example $\prod_{p\mid n}$ means a product where the index $p$ runs over the prime divisors of $n$. 

We denote by $C_n$, a cyclic group of order $n$,
by $S_n$ and $A_n$ the symmetric and alternating groups on $n$ symbols and by $D_n$ and $Q_n$, the dihedral and quaternion groups of order $n$.

 All through, $G$ is a finite group.
 We use the following common notation: $|g|$ denotes the order of $g$, $g^h$ denotes $h^{-1}gh$ and $[g,h]$ denotes the commutator $g^{-1}g^h$ for $g,h\in G$. Moreover, let $\pi(G)$ denote the prime spectrum of $G$, i.e., the set of primes dividing $|G|$ and let $\Aut(G)$ be the group of automorphisms of $G$. The center of $G$ is denoted by $\Z(G)$ and $G'$ denotes the commutator subgroup of $G$. For a subset $X$ of $G$, $\N_G(X)$ and $\C_G(X)$  respectively denote the normalizer and centralizer of $X$ in $G$. For $p\in \pi(G)$, $G_p$ denotes a Sylow $p$-subgroup of $G$. Finally, for a set $\pi$ of primes, a Hall $\pi$-subgroup of $G$ and a Hall $\pi'$-subgroup of $G$ will be denoted by $G_{\pi}$ and $G_{\pi'}$ respectively.
 
By Hall's Theorem (see e.g. \cite[9.1.7]{Robinson}), a solvable group contains a Hall $\pi$-subgroup for each set of primes $\pi$ and any two Hall $\pi$-subgroups are conjugate in $G$, in particular isomorphic, showing that $G_\pi$ is unique up to isomorphism in solvable groups $G$.
A section of $G$ is a group of the form $M/N$ where $M$ is a subgroup of $G$ and $N$ is a normal subgroup of $M$. 
In case $M$ and $N$ are normal in $G$ then we say that $M/N$ is a normal section of $G$. 
Any semi-direct product $G\rtimes H$ is assumed to be not a direct product. Moreover, $\SG{n,m}$ denotes the  $m$-th group or order $n$ in the Small Groups Library of \textsf{GAP} \cite{GAP4}.

If $\Gamma$ is a graph then $v\in \Gamma$ means that $v$ is a vertex of $\Gamma$ and $v-w \in \Gamma$ means that $v$ and $w$ are different vertices of $\Gamma$ joined by an edge.

\subsection{Fitting subgroups and Fitting series} Recall that the \emph{Fitting subgroup $\F(G)$} of the finite group $G$ is the largest nilpotent normal subgroup of $G$ and as such the direct product of the largest normal $p$-subgroups $\O_p(G)$ of $G$.

The \emph{Fitting series} 
\[1 = \F_0(G) \leqslant \F_1(G) \leqslant \F_2(G) \leqslant ... \] of $G$ 
is defined by 
$$\F_0(G) = 1 \qand \frac{\F_{j}(G)}{\F_{j-1}(G)}=\F\left(\frac{G}{F_{j-1}(G)}\right) \text{ for } j\geqslant 1.$$ 
Observe that if $G$ is solvable and $G\ne 1$ then $\F(G)\ne 1$.
Hence, as $G$ is finite it follows that $G$ is solvable if and only if $\F_n(G)=G$ for some $n$. 
The minimal $n$ satisfying this is called the \emph{Fitting length} of $G$, denoted by $\ell_\F(G)$.

The diameter of a graph is the maximum distance between pairs of vertices where the distance between two vertices in a graph is the number of edges in a shortest path connecting them. Of course a non-connected graph has infinite diameter. Lucido proved that the diameter of a connected component of the GK-graph of a solvable group can be at most $3$ \cite[Corollary~2]{LucidoDiameter}. In the extreme situation she obtained tight restrictions on the Fitting structure of the group which will turn out very useful for us.

\begin{proposition}[{Lucido, \cite[Proposition~3]{LucidoTree}}] \label{DiameterFL}
	Let $G$ a finite solvable group such that $\GK(G)$ has diameter $3$. Then either $\ell_\F(G) \leqslant 3$ or $G$ has a normal section isomorphic to $2.S_4$ and $\ell_\F(G)=4$, where the group $2.S_4$, the non-split covering of $S_4$, is given by the following presentation:  
	\begin{equation}\label{Group2S4} 
	2.S_4=\SG{48,28}=\GEN{a,b,x \mid a^8=a^4b^2=x^3=1,\ a^b=a^{-1},\ b^x=a^2,\ a^x=a^3x^{-1}}.
	\end{equation}
\end{proposition}

\subsection{Frobenius and 2-Frobenius groups}\label{Frobenius}
A finite group $G$ is said to be Frobenius if it contains a non-trivial subgroup $H$ such that $H\cap H^x=1$ for every $x\in G\setminus H$.
The subgroup $H$ is unique up to conjugacy and it is called the \emph{Frobenius complement} of $G$.
Moreover $G$ is Frobenius with complement $H$ if and only if $G=N\rtimes H$ for a normal subgroup $N$ such that $N\cap H=1$ and $\C_G(x)\subseteq N$ for every $x\in N\setminus \{1\}$ (see e.g. \cite[Theorem~11.4.1]{JdR1}).
The subgroup $N$ is unique and it is also called the \emph{Frobenius kernel} of $G$. Moreover the orders of the kernel and the complement are coprime and the order of every element of $G$ divides either the order of the kernel or of the complement. 
By a result of Thompson \cite{Thompson59}, the Frobenius kernel $N$ of a Frobenius group $G$ is nilpotent and hence it follows that $N=\F(G)$, see also \cite[10.5.6]{Robinson}. 
Finite Frobenius \cut groups are completely classified.

\begin{theorem}[{\cite[Theorem~1.3]{Bac}}]\label{FrobeniusCut}
	Every finite Frobenius \cut group is of one of the following forms:
	\begin{enumerate}
		\item $C_3^n\rtimes C_2$, $ C_3^{2n}\rtimes C_4$, $ C_3^{2n}\rtimes Q_8$, 
		$C_5^n\rtimes C_4$, $C_7^n\rtimes C_6$, $C_7^{2n}\rtimes (Q_8\times C_3)$ with $n$ a positive integer, 
		\item $C_5^2\rtimes Q_8$, $C_5^2\rtimes (C_3\rtimes C_4)$, $C_5^2\rtimes \SL(2, 3)$, 
		$C_7^2\rtimes \SL(2,3)$.
		\item $N\rtimes C_3$ with $N$ an extension of two abelian groups of exponent dividing $4$ or a metabelian group of exponent $7$. 
	\end{enumerate} 
	Furthermore, for the  cases in (1) and (2) there is a unique action which makes the group a Frobenius \cut group.
\end{theorem}

The notation $G = N\rtimes_\Frob K$ indicates that $G$ is a Frobenius group that is determined uniquely up to isomorphism by the isomorphism types of its kernel $N$ and complement $K$.
\medskip

A finite group $G$ is called \emph{$2$-Frobenius} if it contains normal subgroups $N \leqslant K$, with $G/N$ and $K$ Frobenius groups with kernels $K/N$ and $N$, respectively. 
Using Thompson's Theorem it follows that $G$ is $2$-Frobenius if and only if $\F_2(G)$ and $G/\F_1(G)$ are Frobenius. Then $\F_1(G)$ and $G/\F_2(G)$ are called the lower kernel and the upper complement of $G$, respectively. 
Furthermore, $\F_2(G)/\F_1(G)$ is called both the upper kernel and the lower complement of $G$. 
We immediately see that 
$$\gcd([\F_2(G):\F_1(G)],[G:\F_2(G)])=\gcd([\F_2(G):\F_1(G)],|\F_1(G)|)=1.$$
Moreover we will need the following:

\begin{lemma}[{\cite[Lemma 2.1]{GKLNS}}]\label{GKLNS}
	Let $G$ be a finite 2-Frobenius group. Then $G/ \F_2(G)$ is cyclic, $\F_2(G) / \F_1(G)$ is cyclic of odd order, and $\F_1(G)$ is not a cyclic group.
\end{lemma}

The following theorem due to Gruenberg and Kegel appeared as Corollary to Theorem A in \cite{Williams}.

\begin{theorem}\label{ConnectedComponents} If $G$ be a finite solvable group, then $\GK(G)$ has at most 2 connected components, and has exactly 2 components, if and only if $G$ is a Frobenius group or a  $2$-Frobenius group.
\end{theorem}

\subsection{Rational groups and \cut groups}\label{SubsectionRationalCut}

As mentioned in the introduction, \cut groups, also known as inverse semi-rational groups contain the class of rational groups. 
For the convenience of the reader, we  define basic terminology and include alences used in the article. An element $g \in G$ is said to be \emph{rational in $G$}, if the conjugacy class of $g$ in $G$ contains all the generators of $\GEN{g}$. Moreover, $g$ is \emph{inverse semi-rational} if every generator of $\GEN{g}$ is conjugate in $G$ to $g$ or $g^{-1}$. The group $G$ is called rational if every element of $G$ is rational in $G$ and is called \cut if every element of $G$ is inverse semi-rational in $G$.

For $g\in G$, set $B_G(g)=\N_G(\GEN{g})/\C_G(g)$. The map associating $x\in \N_G(\GEN{g})$ with the automorphism of $\GEN{g}$ mapping $g$ to $g^x$ induces an injective homomorphism $\iota_g \colon B_G(g)\rightarrow \Aut(\GEN{g})$. 
Thus $B_G(g)$ is canonically isomorphic to a subgroup of $\Aut(\langle g \rangle)$. 
Clearly, $g$ is rational in $G$ if and only if $\iota_g$ is surjective; equivalently, if $|B_G(g)|=\varphi(|g|)$.
Moreover, $g$ is inverse semi-rational but not rational in $G$, if and only if the image of $\iota_g$ has index $2$ in $\Aut(\GEN{g})$ and does not contain the inversion map; equivalently, if $|B_G(g)|=\varphi(|g|)/2$ and $g$ and $g^{-1}$ are not conjugate in $G$.
The above equivalences are collected in the following proposition:
\begin{proposition}
The following conditions are equivalent for a finite group $G$:
\begin{enumerate}
 \item $G$ is \cut (respectively, rational).
 \item Every element of $G$ is inverse semi-rational (respectively, rational).
 \item For every $g\in G$ we have that either $|B_G(g)|=\varphi(g)$ or $|B_G(g)|=\frac{\varphi(g)}{2}$ and $g$ and $g^{-1}$ are not conjugate in $G$ (respectively, $|B_G(g)|=\varphi(g)$ for every $g\in G$).
 \item For every irreducible character of $G$, we have $\chi(G)\subseteq F$ for some imaginary quadratic extension $F$ of $\mathbb{Q}$ (respectively, for $F=\mathbb{Q}$).
\end{enumerate}
\end{proposition}

One readily observes that an element of order at most $2$ is rational and every element of order $3$, $4$ or $6$ is inverse semi-rational. {In particular, every group of exponent dividing $4$ or $6$ is \cut.} 
We now collect some more elementary facts about rational and inverse semi-rational elements in the following two lemmas.

\begin{lemma}\label{PrimePower}
	Let $G$ be a finite group and let $g$ be an element of $G$. Suppose that $|g|$ is $p^n$ or $2p^n$ for an odd prime $p$.
	\begin{enumerate}

		\item If $p\equiv 1 \mod 4$ then $g$ is rational in $G$ if and only if $g$ is inverse semi-rational in $G$ if and only if $B_G(g)$ contains an element of order $p^{n-1}(p-1)$.
		\item If $p\equiv -1 \mod 4$ then $g$ is inverse semi-rational (respectively, rational) if and only if $p^{n-1}\frac{p-1}{2} \leqslant |B_G(g)|$ (respectively, $|B_G(g)|=p^{n-1}(p-1)$) if and only if $B_G(g)$ contains an element of order $p^{n-1}\frac{p-1}{2}$ (respectively, of order $p^{n-1}(p-1)$). 
	\end{enumerate}

\end{lemma}	

\begin{proof}
The assumption implies that $\Aut(\GEN{g})$ is cyclic of order $p^{n-1}(p-1)$, so that inversion is its unique element of order $2$. Observe that the subgroup of index $2$ in $\Aut(\GEN{g})$ contains the inversion map if and only if $p \equiv 1 \bmod 4$. Both assertions now follow from the discussion preceding this lemma.
\end{proof}

The following lemma collects some elementary facts about rational and \cut groups.

\begin{lemma}\label{ElementaryDP}
	Let $G$ and $H$ be finite groups and let $g\in G$ and $h\in H$.
	\begin{enumerate} 		
		\item\label{RationalTimesRational} If $g$ is rational in $G$ and $h$ is rational in $H$ then $(g,h)$ is rational in $G\times H$.
		\item\label{RationalTimesSR} If $g$ is rational in $G$ and $h$ is inverse semi-rational in $H$ then $(g,h)$ is inverse semi-rational in $G\times H$.	
		\item If $G$ is rational (respectively, \cut) then so is every epimorphic image of $G$.
		\item\label{RationalTimes} $G\times H$ is rational if and only if so are $G$ and $H$.
		\item\label{RationalTimesCut} If $G$ is rational and $H$ is \cut then $G\times H$ is \cut.
%
%
	\end{enumerate}
\end{lemma}


It is well known that $G$ is rational if and only if every irreducible character of $G$ takes values in $\QQ$ \cite[Problem (2.12)]{Isa06}.
Furthermore, $G$ is \cut if and only if for every irreducible complex character $\chi$ of $G$ there is an imaginary quadratic extension of $\QQ$ containing $\chi(G)$ (\cite{RS}, see also \cite[Corollary~7.1.2 and Corollary~7.1.15]{JdR1}). For future use we include here a GAP function which checks whether a group is \cut, where the input can be either a group or its character table:

\begin{verbatim}
IsCutGroup := function(C)
  return ForAll(Irr(C), chi ->
    Field(chi) = Rationals 
    or 
    Dimension(Field(chi)) = 2 and ImaginaryPart(PrimitiveElement(Field(chi))) <> 0 
  );
end;
\end{verbatim}

\subsection{Solvable \cut groups} Being a \cut group is quite a restrictive property for  a finite solvable group. We collect some facts about solvable \cut groups.

Note that while every prime divides the order of some \cut group, as all symmetric groups are rational, the primes dividing the order of a solvable \cut group are very limited. More precisely, we have the following:

\begin{theorem}[{\cite[Theorem~1.2]{Bac}}]\label{2357} 
	If $G$ is a finite solvable \cut group then $\pi(G)\subseteq \{2,3,5,7\}$.
\end{theorem}

Observe that by \Cref{PrimePower}, if $g$ is an inverse semi-rational element of $G$ of order $5$ (respectively $7$) then $B_G(g)$ has an element of order $4$, (respectively $3$). Therefore, if $G$ is $\cut$ and $5\in \pi(G)$ (respectively, $7\in \pi(G)$) then $G$ has an element of order $4$ (respectively 3) and hence $2\in \pi(G)$ (respectively, $3\in \pi(G)$). Combining this with \Cref{2357} we obtain the following:

\begin{lemma}\label{PrimeSpectrum}
	If $G$ is a finite solvable \cut group, then $\pi(G)$ is one of the following sets: $$\{2\},~\{3\},~\{2,3\},~\{2,5\},~\{3,7\},~\{2,3,5\},~\{2,3,7\},~\{2,3,5,7\}.$$       
\end{lemma}

That each of the sets in \Cref{PrimeSpectrum} actually does appear as $\pi(G)$ for some solvable \cut group $G$ follows from the next proposition, that provides examples of solvable \cut groups realizing GK-graphs appearing in Theorems~\ref{TLess4} and \ref{T4}.

\begin{proposition}\label{PGroups}
The graphs in \Cref{fig:Examples} are realized by the solvable \cut groups listed along with them. 
In addition, if the given group realizing the graph is even rational then the graph and the group have
a shaded background.

	\begin{figure}[h]
		\begin{tabular}{|c|}
			\hline
			\emph{(a)} \begin{subfigure}{.15\textwidth}
	\centering
	\begin{tikzpicture}[background rectangle/.style={fill=gray!45}, show background rectangle]
	\draw (0.5,1.5) node{$C_2$};
	\node[label=east:{$2$}] at (0.5,1) (3){};
	\foreach \p in {3}{
		\draw[fill=black] (\p) circle (0.075cm);
	}
	\end{tikzpicture}
\end{subfigure} 
			\hspace{3cm} 
			\emph{(b)} \begin{subfigure}{.15\textwidth}
				\centering
				\begin{tikzpicture}
				\draw (0.5,1.5) node{$C_3$};
					\node[label=east:{$3$}] at (0.5,1) (3){};
					\foreach \p in {3}{
						\draw[fill=black] (\p) circle (0.075cm);
					}
				\end{tikzpicture}
			\end{subfigure} 
			\\ \hline
			\emph{(c)}
			\begin{subfigure}{.18\textwidth}
				\centering
				\begin{tikzpicture}[background rectangle/.style={fill=gray!45}, show background rectangle,minimum height=1.045cm]
				\draw (0.5,1.5) node{$S_3=C_3\rtimes_\Frob C_2$};
					\node[label=west:{$2$}] at (0,1) (2){};
					\node[label=east:{$3$}] at (0.5,1) (3){};
					\foreach \p in {2,3}{
						\draw[fill=black] (\p) circle (0.075cm);
					}
				\end{tikzpicture}
			\end{subfigure}
			\emph{(d)}\hspace{-0.4cm} 
			\begin{subfigure}{.18\textwidth}
				\centering
				\begin{tikzpicture}[background rectangle/.style={fill=gray!45}, show background rectangle,minimum height=1.045cm]
				\draw (0.5,1.5) node{$S_3 \times C_2$};
					\node[label=west:{$2$}] at (0,1) (2){};
					\node[label=east:{$3$}] at (0.5,1) (3){};
					\foreach \p in {2,3}{
						\draw[fill=black] (\p) circle (0.075cm);
					}
					\draw (2)--(3);
				\end{tikzpicture}
			\end{subfigure} 
			\hspace{-0.4cm}\emph{(e)}\hspace{-0.4cm} 
			\begin{subfigure}{.18\textwidth}
				\centering
				\begin{tikzpicture}[background rectangle/.style={fill=gray!45}, show background rectangle]
				\draw (0.5,1.5) node{$C_5^2 \rtimes_\Frob Q_8$};				
					\node[label=west:{$2$}] at (0.5,1) (2){};
					\node[label=west:{$5$}] at (0.5,0.5) (3){};
					\foreach \p in {2,3}{
						\draw[fill=black] (\p) circle (0.075cm);
					}
				\end{tikzpicture}
			\end{subfigure}
			\hspace{-0.4cm}\emph{(f)}\hspace{-0.1cm} 
			\begin{subfigure}{.20\textwidth}
				\centering

				\begin{tikzpicture}[background rectangle/.style={fill=gray!45}, show background rectangle]
				\draw (0.5,1.5) node{$[C_5^2 \rtimes_\Frob Q_8] \times C_2$};								
					\node[label=west:{$2$}] at (0.5,1) (2){};
					\node[label=west:{$5$}] at (0.5,0.5) (5){};
					\foreach \p in {2,5}{
						\draw[fill=black] (\p) circle (0.075cm);
					}
					\draw (2)--(5);
				\end{tikzpicture}
			\end{subfigure}
			\emph{(g)}\hspace{-0.5cm} 
			\begin{subfigure}{.17\textwidth}
				\centering
				\begin{tikzpicture}
				\draw (0.5,1.5) node{$C_7 \rtimes_\Frob C_3$};
					\node[label=east:{$3$}] at (0.5,1) (3){};
					\node[label=east:{$7$}] at (0.5,0.5) (7){};
					\foreach \p in {3,7}{
						\draw[fill=black] (\p) circle (0.075cm);
					}
				\end{tikzpicture}
			\end{subfigure}
			
			\\\hline
			\emph{(h)}\hspace{-0.3cm} 			
			\begin{subfigure}{.2\textwidth}
				\centering

				\begin{tikzpicture}
				\draw (0.5,1) node{$C_5^2 \rtimes_\Frob (C_3 \rtimes C_4)$};				
					\node[label=west:{$2$}] at (0,0.5) (2){};
					\node[label=east:{$3$}] at (0.5,0.5) (3){};
					\node[label=west:{$5$}] at (0,0) (5){};
					\foreach \p in {2,3,5}{
						\draw[fill=black] (\p) circle (0.075cm);
					}
					\draw (2)--(3);
				\end{tikzpicture}
			\end{subfigure}
			\hspace{-0.3cm}\emph{(i)}\hspace{-0.1cm}
			\begin{subfigure}{.22\textwidth}
				\centering
				
				\begin{tikzpicture}
				\draw (0.5,1) node{$[C_5^2 \rtimes_\Frob (C_3 \rtimes C_4)] \times C_2$};								
					\node[label=west:{$2$}] at (0,0.5) (2){};
					\node[label=east:{$3$}] at (0.5,0.5) (3){};
					\node[label=west:{$5$}] at (0,0) (5){};
					\foreach \p in {2,3,5}{
						\draw[fill=black] (\p) circle (0.075cm);
					}
					\draw (2)--(3);
					\draw (2)--(5);
				\end{tikzpicture}
			\end{subfigure}
			\hspace{0.3cm}\emph{(j)}\hspace{-0.4cm}
			\begin{subfigure}{.22\textwidth}
				\centering
				\begin{tikzpicture}
				\draw (0.5,1) node{$[C_5^2 \rtimes_\Frob Q_8] \times C_3$};				
					\node[label=west:{$2$}] at (0,0.5) (2){};
					\node[label=east:{$3$}] at (0.5,0.5) (3){};
					\node[label=west:{$5$}] at (0,0) (5){};
					\foreach \p in {2,3,5}{
						\draw[fill=black] (\p) circle (0.075cm);
					}
					\draw (2)--(3);
					\draw (3)--(5);
				\end{tikzpicture}
			\end{subfigure}
			\hspace{-0.5cm}\emph{(k)}\hspace{-0.2cm}
			\begin{subfigure}{.22\textwidth}
				\centering

				\begin{tikzpicture}[background rectangle/.style={fill=gray!45}, show background rectangle,minimum height=0.45cm]
				\draw (0.5,1) node{$[C_5^2 \rtimes_\Frob Q_8] \times S_3$};								
					\node[label=west:{$2$}] at (0,0.5) (2){};
					\node[label=east:{$3$}] at (0.5,0.5) (3){};
					\node[label=west:{$5$}] at (0,0) (5){};
					\foreach \p in {2,3,5}{
						\draw[fill=black] (\p) circle (0.075cm);
					}
					\draw (2)--(3);
					\draw (2)--(5);
					\draw (3)--(5);
				\end{tikzpicture}
			\end{subfigure}
			\\
			\emph{(l)}\hspace{-0.5cm}
				\begin{subfigure}{.2\textwidth}
				\centering
				\begin{tikzpicture}
				\draw (0.5,1) node{$C_{7} \rtimes_\Frob  C_6$};				
					\node[label=west:{$2$}] at (0,0.5) (2){};
					\node[label=east:{$3$}] at (0.5,0.5) (3){};
					\node[label=east:{$7$}] at (0.5,0) (7){};
					\foreach \p in {2,3,7}{
						\draw[fill=black] (\p) circle (0.075cm);
					}
					\draw (2)--(3);
				\end{tikzpicture}
			\end{subfigure}
			\hspace{-0.2cm}\emph{(m)}\hspace{-0.5cm}
			\begin{subfigure}{.22\textwidth}
				\centering
				\begin{tikzpicture}
				\draw (0.5,1) node{$[C_{7} \rtimes_\Frob  C_6] \times C_2$};								
					\node[label=west:{$2$}] at (0,0.5) (2){};
					\node[label=east:{$3$}] at (0.5,0.5) (3){};
					\node[label=east:{$7$}] at (0.5,0) (7){};
					\foreach \p in {2,3,7}{
						\draw[fill=black] (\p) circle (0.075cm);
					}
					\draw (2)--(3);
					\draw (2)--(7);
				\end{tikzpicture}
			\end{subfigure}
			\hspace{-0.2cm}\emph{(n)}\hspace{-0.5cm}
			\begin{subfigure}{.22\textwidth}
				\centering
				\begin{tikzpicture}
				\draw (0.5,1) node{$[C_{7} \rtimes_\Frob  C_6] \times C_3$};				
					\node[label=west:{$2$}] at (0,0.5) (2){};
					\node[label=east:{$3$}] at (0.5,0.5) (3){};
					\node[label=east:{$7$}] at (0.5,0) (7){};
					\foreach \p in {2,3,7}{
						\draw[fill=black] (\p) circle (0.075cm);
					}
					\draw (2)--(3);
					\draw (3)--(7);
				\end{tikzpicture}
			\end{subfigure}
			\hspace{-0.2cm}\emph{(o)}\hspace{-0.3cm}
			\begin{subfigure}{.2\textwidth}
				\centering
				\begin{tikzpicture}
				\draw (0.5,1) node{$[C_{7} \rtimes_\Frob  C_3] \times S_3$};				
					\node[label=west:{$2$}] at (0,0.5) (2){};
					\node[label=east:{$3$}] at (0.5,0.5) (3){};
					\node[label=east:{$7$}] at (0.5,0) (7){};
					\foreach \p in {2,3,7}{
						\draw[fill=black] (\p) circle (0.075cm);
					}
					\draw (2)--(3);
					\draw (2)--(7);
					\draw (3)--(7);
				\end{tikzpicture}
			\end{subfigure}
			\\\hline
			\hspace{-0.3cm}\emph{(p)}\hspace{-0.3cm}			
			\begin{subfigure}{.28\textwidth}
				\centering
				\begin{tikzpicture}
				\draw (0.5,1) node{$[C_5^2 \rtimes_\Frob Q_8] \times [C_{7} \rtimes_\Frob  C_3]$};				
					\node[label=west:{$2$}] at (0,0.5) (2){};
					\node[label=east:{$3$}] at (0.5, 0.5) (3){};
					\node[label=west:{$5$}] at (0, 0) (5){};
					\node[label=east:{$7$}] at (0.5, 0) (7){};
					\foreach \p in {2,3,5,7}{
						\draw[fill=black] (\p) circle (0.075cm);
					}
					\draw (2)--(3);
					\draw (2)--(7);;
					\draw (3)--(5);
					\draw (5)--(7);
				\end{tikzpicture}
			\end{subfigure}
			\hspace{-0.2cm}\emph{(q)}\hspace{-0.3cm}
			\begin{subfigure}{.28\textwidth}
				\centering
				\begin{tikzpicture}
				\draw (0.5,1) node{$[C_5^2 \rtimes_\Frob Q_8] \times [C_{7} \rtimes_\Frob  C_3] \times C_2$};				
					\node[label=west:{$2$}] at (0,0.5) (2){};
					\node[label=east:{$3$}] at (0.5, 0.5) (3){};
					\node[label=west:{$5$}] at (0, 0) (5){};
					\node[label=east:{$7$}] at (0.5, 0) (7){};
					\foreach \p in {2,3,5,7}{
						\draw[fill=black] (\p) circle (0.075cm);
					}
					\draw (2)--(3);
					\draw (2)--(5);
					\draw (2)--(7);
					\draw (3)--(5);
					\draw (5)--(7);
				\end{tikzpicture}
			\end{subfigure} 
			\hspace{0.4cm}\emph{(r)}\hspace{-0.4cm}
				\begin{subfigure}{.28\textwidth}
					\centering
					\begin{tikzpicture}
				\draw (0.5,1) node{$[C_5^2 \rtimes_\Frob Q_8] \times [C_{7} \rtimes_\Frob  C_6] \times C_3$};					
						\node[label=west:{$2$}] at (0,0.5) (2){};
						\node[label=east:{$3$}] at (0.5, 0.5) (3){};
						\node[label=west:{$5$}] at (0, 0) (5){};
						\node[label=east:{$7$}] at (0.5, 0) (7){};
						\foreach \p in {2,3,5,7}{
							\draw[fill=black] (\p) circle (0.075cm);
						}
						\draw (2)--(3);
						\draw (2)--(5);
						\draw (2)--(7);
						\draw (3)--(5);
						\draw (3)--(7);
						\draw (5)--(7);
					\end{tikzpicture}
				\end{subfigure}
			\\\hline
		\end{tabular}	\caption{\label{fig:Examples} Realization of GK-graphs by solvable \cut\ groups.}
	\end{figure}
\end{proposition}

\begin{proof}
We have to prove that the group in (a)-(r) are \cut and those in (a), (c), (d), (e), (f) and (k) are rational.
	
It is well known that the groups $C_2$, $S_3$ and $C_5^2\rtimes_{\Frob} Q_8$ are rational. By Lemma~2.7(4), direct product of rational groups are rational and hence so are the groups in (a), (c), (d), (e), (f) and (k).
	
Clearly $C_3$ is \cut and, by Theorem~2.2, so are the groups in (g), (h) and (l).
	Let $G$ be the group in (n) and let $g\in G$. Then either $|g|$ divides $6$ or $|g|\in \{7,3\cdot 7\}$. If $|g|$ divides $6$, then it is clearly inverse semi-rational in $G$. 
	Suppose that $|g|=3\times 7$. Then $g_3$ is central in $G$ and $\GEN{g_7}$ is the unique subgroup of order $7$ of $G$. Moreover $G$ has an element $h$ such that $g_7^h=g_7^3$. Then $g$ is inverse semi-rational in $G$ and $g_7$ is rational in $G$. This shows that $G$ is \cut. 
	
So all the groups in (a), (c), (d), (e), (f) and (k) are rational and the groups in (b), (g), (h), (l) and (n) are \cut. The remaining groups, i.e. those in (i), (j), (m), (o), (p), (q) and (r), are the direct product of one of the latter by one or two of the former and hence they are \cut, by \Cref{ElementaryDP}\eqref{RationalTimes}.
\end{proof}

\section{GK-graphs of solvable \cut groups with at most 3 vertices}\label{SectionTLess4}

In this section, we prove \Cref{TLess4}, i.e.\ we prove which graphs with at most $3$ vertices are GK-graphs of solvable \cut\ groups. We also discuss which of these graphs can be realized as the GK-graph of a solvable rational group and prove \Cref{PrimeGraphRat}.
		
We begin with a result of Higman which states that the number of primes dividing the order of a finite solvable group in which the order of every element is a prime power is at most $2$ \cite{Hig57a}. This result implies the following result, which was independently proved with different methods by Lucido \cite[Proposition~1]{LucidoDiameter} and is known as ``Lucido's Three Primes Lemma''.

\begin{lemma}[Higman, Lucido]\label{Higman3Vertices}
	If $X$ is a set containing $3$ vertices of the GK-graph $\Gamma$ of a finite solvable group, then at least two elements of $X$ are joined by an edge in $\Gamma$. 
\end{lemma}

	The following lemma will be very useful in this paper and will be used frequently throughout. 

\begin{lemma}\label{6and35}
	Let $\Gamma$ be the GK-graph of a finite \cut group. 
	\begin{enumerate}
		\item If $2-7\in \Gamma$, $3-5 \in \Gamma$, $3-7\in \Gamma$ or $5-7\in \Gamma$, then $2-3\in \Gamma$.
		\item If $5-7\in \Gamma$, then $2-7\in \Gamma$ and $3-5 \in \Gamma$. 
	\end{enumerate}
\end{lemma}

\begin{proof} Let $G$ be a \cut group and let $g\in G$. 
	
(1)	If $|g|=2\cdot 7$ then $B_G(g)$ has an element of order $\varphi(|g|)/2=3$, since $g$ is inverse semi-rational in $G$ and hence $G$ contains an element of order $3$  commuting with the $2$-part of $g$. Thus $G$ contains an element of order $2\cdot 3$.

If $|g|=3\cdot 5$ then $\Aut(\GEN{g})$ has an element $\alpha$ of order $4$ acting as the identity on the $3$-part of $g$. Then $\alpha^2$ is an element of order $2$ in the image of $\iota_g$. 
This implies that $G$ has a $2$-element commuting with the $3$-part of $g$ and hence $G$ has an element of order $6$.

If $|g|$ is either $3\cdot 7$ or $5\cdot 7$ then $|B_G(g)|$ is a multiple of $6$ because $\frac{\varphi(21)}{2}=6$ and $\frac{\varphi(35)}{2}=12$. As $B_G(g)$ is abelian we deduce that $G$ contains an element of order $6$. 
	
(2) If $|g|=5\cdot 7$ then $\Aut(\GEN{g})$ has an element $\alpha$ of order $4$ commuting with the $7$-part of $g$ and an element $\beta$ of order $6$ commuting with the $5$-part of $g$. As $\alpha^2$ and $\beta^2$ belong to the image of $\iota_g$, we deduce that $G$ contains a $2$-element commuting with an element of order $7$ and a $3$-element commuting with the $5$-part of $G$. Thus $G$ contains an element of order $2\cdot 7$ and an element of order $3\cdot 5$. \qedhere
\end{proof}

\noindent\textbf{\textit{Proof of \Cref{TLess4}.} } The sufficiency part of \Cref{TLess4} follows from \Cref{PGroups}. Let $G$ be a solvable \cut group such that its GK-graph $\Gamma $ has at most $3$ vertices. The possible sets of vertices are described in Lemma~\ref{PrimeSpectrum}. Because of \Cref{Higman3Vertices}, the graphs with three vertices and no edges are excluded. By Lemma~\ref{6and35}, it only remains to prove that $\Gamma\ne (2-5\hspace{.4cm}3)$. Now, if $\Gamma = (2-5\hspace{.4cm}3)$, then in view of \Cref{ConnectedComponents}, $G$ is either a Frobenius or a $2$-Frobenius group. Inspecting the list of Frobenius \cut~groups in Theorem~\ref{FrobeniusCut}, we observe that $G$ cannot be Frobenius. Hence, $G$ is $2$-Frobenius implying that $\F_2(G)$ is Frobenius and $G/\F_1(G)$ is Frobenius as well as \cut. Moreover, the upper kernel $\F_2(G)/\F_1(G)$ is cyclic of odd order and the upper complement $G/\F_2(G)$ is cyclic by Lemma~\ref{GKLNS}. Then, $G/\F_1(G)$ is either $S_3$ or $C_5\rtimes C_4$ by Theorem~\ref{FrobeniusCut}. However,  the latter is not compatible with the assumption that $G$ has elements of order $2\cdot 5$. Thus $G/\F_1(G)\simeq S_3$ and hence $\F_1(G)$ has an element $g$ of order $5$. 
By \Cref{PrimePower}, $g$ is rational and hence $G$ has an element $h$ of order $4$ such that $[g,h^2]\ne 1$. However, as $G/\F_1(G)\simeq S_3$ it follows that $h^2\in \F_1(G)$. As $\F_1(G)$ is nilpotent we have $[g,h^2]=1$, a contradiction.  \hfill  $\Box$\\

In case $G$ is a solvable rational group then $\pi(G)\subseteq \{2,3,5\}$ by a classical result of Gow \cite{Gow} and $\GK(G)$ is one of those in Figure~\ref{Less4}. 
In the remainder of the section, we prove \Cref{PrimeGraphRat}. 
We first need some lemmas.

 \begin{lemma}	\label{MANS}
	Let $\mathcal{G}$ be a family of finite solvable groups, which is closed under epimorphic images and suppose that the GK-graph of any group  $G\in \mathcal{G}$ with $m$ vertices has at least $k$ edges. If $G$ is a group of minimal order among those in $\mathcal{G}$ with $m$ vertices in its GK-graph and $A$ is a minimal normal subgroup of $G$, then $A$ is an elementary abelian Sylow subgroup of $G$. 
\end{lemma}
\begin{proof}
	As $G$ is solvable, $A$ is an elementary abelian $p$-group and by the minimality assumption on $G$, $\GK(G/A)$ is properly contained in $\GK(G)$. As $\mathcal{G}$ is closed under epimorphic images, $\GK(G/A)$ has fewer vertices than $\GK(G)$ so that $A$ is a Sylow $p$-subgroup of $G$. 
\end{proof}

\begin{lemma}\label{GpCyclic}
 If $G$ is a finite \cut group and $G_p$ is an abelian Sylow $p$-subgroup of $G$ then the exponent of $G_p$ divides $p$ or $4$.
\end{lemma}

\begin{proof}
	Let $x$ be an element of $G_p$ of order $p^\alpha$. 
	As $x$ is inverse semi-rational, $|B_G(x)|$ is either $p^{\alpha-1}(p-1)$ or $\frac{p^{\alpha-1}(p-1)}{2}$. Therefore, if $p\neq 2$ and $\alpha >1$ or if $p=2$ and $\alpha >2$, then $p|[\N_G(\GEN{x}):\C_G(x)]$, which is not possible, as $G_p\subseteq \C_G(x)$. Thus if $p\ne 2$ then $\alpha\leqslant 1$ and if $p=2$ then $\alpha\leqslant 2$.
\end{proof}

\begin{lemma}\label{CyclicSylow} 
	Let $G$ be a finite \cut group and let $G_p$ and $G_q$ be Sylow subgroups of $G$ for two distinct primes $p$ and $q$ dividing the order of $G$. Suppose that $G_p$ is normal in $G$ and $G$ does not contain an element of order $pq$. Then $G_q$ is either the quaternion group of order $8$ or a cyclic group of order dividing $4$ or $q$.
\end{lemma}

\begin{proof} Set $F = G_pG_q$. Since there is no element of order $pq$ in $G$, $G_q$ acts fixed point-freely by conjugation on $G_p$. 
Hence, $F$ is a Frobenius group with Frobenius kernel $G_p$ and Frobenius complement $G_q$. From the known structure theory of Frobenius complements it follows that $G_q$ is cyclic or generalized quaternion, see \cite[10.5.6]{Robinson}.
	
By \Cref{GpCyclic}, if $G_q$ is cyclic then its order is a divisor of $4$ or $q$. Finally, if $G_2$ is generalized quaternion of order $2^f$, then $G_2$ contains a normal cyclic subgroup of order $2^{f-1}$ generated by $y$, say. $G_2$ induces on $\langle y \rangle$ the inversion automorphism, $B_G(y)$ has order not divisible by $4$ and since $y$ is inverse semi-rational in $G$ it follows that $f - 1 \leqslant 2$ and hence $G_2$ is the quaternion group of order $8$. \end{proof} 
 
We are ready for the proof of \Cref{PrimeGraphRat}, describing GK-graphs of solvable rational groups.
\medskip

\noindent
\textbf{\textit{Proof of \Cref{PrimeGraphRat}.}} As stated above, for a finite solvable rational group $G$, we have $\pi(G)\subseteq \{2,3,5\}$ \cite{Gow}. Moreover, if $G\ne 1$ then $2\in \pi(G)$. So, if $\GK(G)$ has at most $2$ vertices, then in view of the restrictions on the prime spectra of $G$, the only possibilities of GK-graphs of $G$ are (a), (c)-(f) and by \Cref{PGroups}, these GK-graphs are indeed realized by solvable rational groups. 
	
Now, consider the case when $\pi(G)=\{2,3,5\}$. By \Cref{PGroups}, the complete graph on the vertices $\{2,3,5\}$ is the GK-graph of a solvable rational group. Hence, in order to prove the stated theorem, we need to show that the only other possibility of the GK-graph of $G$, say $\Gamma$, is the connected graph with edge $3-5$ missing. 
	
 The graph with no edge is already ruled out by \Cref{Higman3Vertices}. If $\Gamma$ has exactly one edge, i.e., $2$ connected components, then by \Cref{ConnectedComponents}, $G$ must be a Frobenius or a $2$-Frobenius group. But, by the classification of rational Frobenius groups \cite{DS}, it follows that $G$ cannot be a Frobenius group and in view of the fact that the order of a $2$-Frobenius rational group is not divisible by $5$ \cite[Lemma 4]{DIM}, $G$ cannot be a $2$-Frobenius group. Therefore, $\Gamma$ must have at least two edges. Furthermore, by \Cref{6and35}, $\Gamma$ cannot be $(3-5-2)$. So, it only remains to prove that $\Gamma\neq (2-3-5)$.

 Let $G$ be a minimal solvable rational group that realizes the graph $\Gamma = (2-3-5)$ and let $A$ be a minimal normal subgroup of $G$. By \Cref{MANS}, $A$ is an elementary abelian Sylow $p$-subgroup of $G$, with $p\in\{3,5\}$. If $p=3$, then $\GK{(G/A)}=(2 ~~5)$, and hence,  by \Cref{ConnectedComponents}, $G/A$ must either be a Frobenius group or a 2-Frobenius group. The latter option is not
 feasible since $5$ does not divide the order of a 2-Frobenius rational group. Hence, $G/A$ must be Frobenius and therefore isomorphic to $C_5^2\rtimes Q_8$, by {\cite{DS}}. Then the Sylow $2$-subgroup of $G$ is $Q_8$. If $p=5$,  then \Cref{CyclicSylow} yields that the Sylow $2$-subgoup of $G$ is either $C_4$ or $Q_8$. 
As $G$ has a rational element $g$ of order $3\cdot 5$, which implies that $B_G(g)\simeq C_4\times C_2$, both cases lead to a contradiction. \hfill  $\Box$\\

\section{GK-graphs of solvable \cut groups with $4$ vertices}\label{Section4}

In this section, we prove \Cref{T4} which restricts the possibilities of GK-graphs of solvable \cut groups with more than  three vertices. The first part follows from \Cref{PGroups}. To prove the second part, we firstly prove that if $\Gamma$ is a GK-graph of a solvable \cut group and has four vertices, then $\Gamma$ has at least three edges, and secondly that it cannot have exactly three edges. Then we will complete the proof excluding one by one the graphs with at least four edges not appearing in \Cref{Four}.\\

All throughout this section $G$ is a
solvable \cut group with $|\pi(G)|>3$. Then $\pi(G) = \{2,3,5,7\}$, by \Cref{2357}.

\subsection{Excluding graphs with less than $3$ edges} 
In this subsection, we prove that $\Gamma$ has at least three edges. 
We start excluding Frobenius and 2-Frobenius groups from our discussion: 

\begin{proposition}\label{Frobenius3Vertices} 
	The order of a finite \cut group which is a Frobenius group or a $2$-Frobenius {group} is divisible by at most 3 primes.
\end{proposition}

\begin{proof} If $G$ is Frobenius and \cut then the result follows directly from Theorem~\ref{FrobeniusCut}.
	So suppose that $G$ is a $2$-Frobenius \cut group with $|\pi(G)|>3$. 
	By \Cref{2357}, $\pi(G) = \{2,3,5,7\}$, as $2$-Frobenius groups are always solvable. 
	By Lemma~\ref{GKLNS}, $G/\F(G)$ is a Frobenius group with cyclic kernel of odd order and cyclic complement. 
	Using again Theorem~\ref{FrobeniusCut} we have that $G/\F(G)$ is isomorphic to one of the following four Frobenius groups: 
	\[C_3 \rtimes C_2 \simeq S_3,\quad C_5 \rtimes C_4,\quad C_7 \rtimes C_3,\quad \text{or}\quad C_7\rtimes C_6. \]
	
	Assume first that $G/\F(G) \simeq C_5 \rtimes C_4$. Then $\F(G)$ is a nilpotent group of order divisible by $7$ and hence there is an element $x$ in the center of $\F(G)$ of order $7$. 
	Thus $B_G(x) = \N_G(\langle x \rangle) / \C_G(x)$ has to be isomorphic to a quotient of $G/\F(G)$, which is a $3'$ group. 
	Since $|\Aut(\langle x\rangle)|=6$, we get that $\left[\Aut(\langle x \rangle) : B_G(x) \right] > 2$, contradicting the fact that $G$ is a \cut group.
	
	Similarly, if $G/\F(G) \simeq S_3$ or $C_7 \rtimes C_3$ or $C_7 \rtimes C_6$, then there is an element $x$ of order 5 in the center of $\F(G)$.
Reasoning as above, we can prove that $|B_G(x)|$ is not divisible by $4$, so that  $x$ is not inverse semi-rational by \Cref{PrimePower}, a contradiction.
\end{proof}

As a direct consequence of Theorem~\ref{ConnectedComponents} and Proposition~\ref{Frobenius3Vertices}, we obtain the following 

\begin{corollary}\label{3ImpliesConnected}
	The GK-graph of a finite solvable \cut group with more than 3 vertices is connected.
\end{corollary}

By \Cref{3ImpliesConnected}, $\Gamma$ has exactly one connected component and as $|\pi(G)| =4$, necessarily $\Gamma$ has at least three edges.

\subsection{Excluding graphs with exactly $3$ edges}
In this subsection, we prove that $\Gamma$ has at least four edges. We first prove the following lemma which holds for an arbitrary finite \cut group.

\begin{lemma}\label{G2G3Exclusions}
	Let $G$ be a finite \cut group. 
	\begin{enumerate}
		\item\label{Not4pG2Cyclic} Assume $G_2$ is cyclic and $p$ is an odd prime. Then $G$ has no elements of order $4p$. Furthermore if $p\equiv 1 \mod 4$ then $G$ has no elements of order $2p$.	
		\item\label{Not2pQuaternion} If $G_2$ is quaternion of order $8$ then $G$ has no elements of order $2p$ with $p$ prime and $p\equiv 1 \mod 4$.
		\item\label{No21} If  $G_3$ is cyclic, then $G$ has no elements of order $3\cdot 7$.
	\end{enumerate}  
\end{lemma}

\begin{proof}Let $G$ be a \cut group.

		(1)  Suppose first that $G$ has an element $g$ of order $4p$ and $G_2$ is cyclic. Then it follows from \Cref{GpCyclic} that $G_2 \simeq C_4$. Moreover, as  $\C_G(g)$ contains a Sylow $2$-subgroup of $G$, $|B_G(g)|$ is odd. This contradicts the fact that $g$ is inverse semi-rational in $G$ as clearly $|\Aut(\GEN{g})|$ is a multiple of $4$. Suppose now that 
		 $G$ has an element $g$ of order $2p$ and $p\equiv 1 \mod 4$. Then by \Cref{PrimePower}\,(1), $g$ must be rational in $G$, i.e. $ |\Aut(\GEN{g})|=|B_G(g)|$ and by \Cref{GpCyclic}, $|G_2|$ is either $2$ or $4$. But as $g$ has order $2p$, we have that $4\nmid |B_G(g)|$ and $4\mid|\Aut(\GEN{g})|$, a contradiction.
		 
		(2)  If $G_2$ is quaternion of order $8$ and $g \in G$ has order $2p$ then the Sylow $2$-subgroup of $B_G(g)$ is elementary abelian, so if $p\equiv 1 \mod 4$ then $g$ is not rational, in contradiction with \Cref{PrimePower}\,(1).
		
		(3) If {$G_3$} is non-trivial and cyclic, then $G_3 \simeq C_3$ by \Cref{GpCyclic}. If $g$ has order $3\cdot 7$ then both $|\C_G(g)|$ and $|B_G(g)|$ are multiples of $3$ and thus $|G_3|$ is a multiple of $9$.
\end{proof}

\begin{proposition}\label{4Vertices3Edges}
The GK-graph of a finite solvable \cut group with four vertices has at least four edges. 
\end{proposition}

\begin{proof} Let $G$ be a finite solvable \cut group of minimal order such that $\Gamma=\GK(G)$ has $4$ vertices and $3$ edges. By \Cref{MANS}, for some prime $p\in \pi(G)$, $G_p$ is a normal subgroup of $G$ which is elementary abelian, so that $G/G_p$ is a solvable \cut group with $|\pi(G/G_p)|=3$. Hence, by \Cref{TLess4},  $\Gamma_{GK}(G/G_p)$ is one of the $8$ graphs in the bottom box of Figure~\ref{Less4}. This implies $p\in \{5,7\}$ and $2-3\in \Gamma.$ Moreover, $5-7\not\in \Gamma$ for, if $5-7\in \Gamma$, then by \Cref{6and35}, $2-3,2-7,3-5\in \Gamma$, contradicting the assumption that $\Gamma$ has only $3$ edges. Furthermore, in view of \Cref{3ImpliesConnected}, $\Gamma$ does not contain a triangle. This reduces the possibilities for $\Gamma$ to just four graphs, namely $\Gamma_1$, $\Gamma_2$, $\Gamma_3$ and $\Gamma_4$ as listed in \Cref{Fig4Graphs}.

	\begin{figure}[ht!]
		\begin{tabular}{ccccccc}
			
				\begin{subfigure}{.15\textwidth}
				\begin{tikzpicture}
				\node[label=west:{$2$}] at (0,1.5) (2){};
				\node[label=east:{$3$}] at (1.5, 1.5) (3){};
				\node[label=west:{$5$}] at (0, 0) (5){};
				\node[label=east:{$7$}] at (1.5, 0) (7){};
				\foreach \p in {2,3,5,7}{
					\draw[fill=black] (\p) circle (.075cm);
				}
				\draw (2)--(3);
				\draw (2)--(7);
				\draw (3)--(5);
				\end{tikzpicture}
			\end{subfigure} 
			& \hspace{1cm} &

		\begin{subfigure}{.15\textwidth}
			\begin{tikzpicture}
			\node[label=west:{$2$}] at (0,1.5) (2){};
			\node[label=east:{$3$}] at (1.5, 1.5) (3){};
			\node[label=west:{$5$}] at (0, 0) (5){};
			\node[label=east:{$7$}] at (1.5, 0) (7){};
			\foreach \p in {2,3,5,7}{
				\draw[fill=black] (\p) circle (.075cm);
			}
			\draw (2)--(5);
			\draw (2)--(3);
			\draw (3)--(7);
			\end{tikzpicture}
		\end{subfigure}
		
		& \hspace{1cm} &	\begin{subfigure}{.15\textwidth}
			\begin{tikzpicture}
			\node[label=west:{$2$}] at (0,1.5) (2){};
			\node[label=east:{$3$}] at (1.5, 1.5) (3){};
			\node[label=west:{$5$}] at (0, 0) (5){};
			\node[label=east:{$7$}] at (1.5, 0) (7){};
			\foreach \p in {2,3,5,7}{
				\draw[fill=black] (\p) circle (.075cm);
			}
			\draw (3)--(5);
			\draw (2)--(3);
			\draw (3)--(7);
			\end{tikzpicture}
		\end{subfigure}& \hspace{1cm} &	\begin{subfigure}{.15\textwidth}
		\begin{tikzpicture}
		\node[label=west:{$2$}] at (0,1.5) (2){};
		\node[label=east:{$3$}] at (1.5, 1.5) (3){};
		\node[label=west:{$5$}] at (0, 0) (5){};
		\node[label=east:{$7$}] at (1.5, 0) (7){};
		\foreach \p in {2,3,5,7}{
			\draw[fill=black] (\p) circle (.075cm);
		}
		\draw (2)--(5);
		\draw (2)--(3);
		\draw (2)--(7);
		\end{tikzpicture}
	\end{subfigure}
			\\
			$\Gamma_1$	& & $\Gamma_2$ &&$\Gamma_3$&&$\Gamma_4$
		\end{tabular}
\caption{}\label{Fig4Graphs}
	\end{figure}
The graphs $\Gamma_3$ and $\Gamma_4$ are not in accordance with \Cref{Higman3Vertices}, with $X=\{2,5,7\}$ and $\{3,5,7\}$ respectively. This leaves us with the {graphs} $\Gamma_1$ and $\Gamma_2$ as the only possible choices of $\Gamma$, which we rule out next. 
	
	Suppose first that $\Gamma=\Gamma_2$. If $p=5$, then by \Cref{CyclicSylow}, $G_3$ is cyclic of order $3$ and hence by 
\Cref{G2G3Exclusions} \eqref{No21}, $G$ has no element of order $3\cdot 7$, a contradiction. If $p=7$, then again by \Cref{CyclicSylow}, $G_2$ is either cyclic of order $2$ or $4$ or isomorphic to $Q_8$. Therefore,  \Cref{G2G3Exclusions}.\eqref{Not4pG2Cyclic}  and \eqref{Not2pQuaternion} imply that $G$ has no element of order $2\cdot 5$, again a contradiction.

Thus, it only remains to prove that $\Gamma\neq \Gamma_1$. Assume that $\Gamma=\Gamma_1$. Observe that as $\Gamma$ does not contain the edge $5-7$, it follows that the Fitting group of $G$ is $G_p$. Then $G/\F(G)$ is a \cut group of order divisible by $5$ or $7$ and hence it cannot be nilpotent \cite[Theorem 1]{BMP17}. Consequently, the Fitting length of $G$ is at least $3$.  Also, as $\Gamma$ has diameter $3$, by \Cref{DiameterFL}, we have that either  $\ell_\F(G)=3$ 
or $\ell_\F(G)=4$ and $G$ has a section isomorphic to the group $2.S_4=\SG{48,28}$ given by the presentation in \eqref{Group2S4}.  Hereforth, we split the details into two cases:

\textbf{Case(1):} $p=5$.

	By \Cref{CyclicSylow}, $G_7\simeq C_7$ and $G_2$ is either a cyclic group of order $4$ or the quaternion group of order $8$.
In any case, $16 \nmid |G|$ and hence $2.S_4$ cannot be a section of $G$. Thus $\ell_\F(G)=3$ and $G_{\{5,7\}}\subseteq \F_2(G)$. As $G$ has no elements of order $3\cdot 7$, $3\nmid |\F_2(G)|$ and applying \Cref{CyclicSylow} to $G/\F(G)$ with $p=7$ and $q=3$ we deduce that  $G_3 \simeq C_3$. 

If $G_2\simeq C_4$ then $G$ has no elements of order $12$ by Lemma~\ref{G2G3Exclusions} and hence $G_{\{2,3\}}$ is the unique non-abelian group of order $12$ with an element of order $4$, namely $C_3\rtimes C_4$. As $\Aut(G_7)$ is abelian, the kernel of the action of $G_{\{2,3\}}$ on $G_7$ contains the commutator subgroup of $G_{\{2,3\}}$ which is $G_3$. This yields an element of order $3\cdot 7$ in $G$, a contradiction.

Thus $G_2\simeq Q_8$. Then $G/G_5$ is a solvable \cut group with Sylow subgroups isomorphic to $Q_8$, $C_3$ and $C_7$. 
Using the GAP function \verb+IsCutGroup+ included in \Cref{SubsectionRationalCut} one can easily check that there is only one group satisfying these conditions: $Q_8\times (C_7\rtimes{_{\Frob}} C_3){=\SG{168,21}}$. 
Since $G/\F(G)$ has a normal Sylow $2$-subgroup, $\F(G)G_2$ is the unique Hall $\{2, 5\}$-subgroup of $G$. Every $5$-element $x$ of $G$ is rational in $G$ and $B_G(x)$ is a $2$-group, so $x$ has to be rational in the unique Hall $\{2, 5\}$-subgroup $\F(G)G_2$ of $G$. Now we may consider $\F(G)$ as a module over $\FF_5G_2=\FF_5Q_8$,  which is a semi-simple algebra. The Wedderburn decomposition of  $\FF_5Q_8$ is $\FF_5 Q_8 \simeq \FF_5 \oplus \FF_5 \oplus \FF_5 \oplus \FF_5 \oplus \operatorname{M}_2(\FF_5)$ and in all $1$-dimensional representation the involution of $G_2 = \GEN{i, j} \simeq Q_8$ acts by multiplication with $1$. Since $G$ does not contain elements of order $2\cdot 5$, $\F(G)$, as $\FF_5 G_2$-module, is a direct sum of $s$ copies of the $2$-dimensional simple $\FF_5 G_2$-module, which corresponds to the $\FF_5$- representation
\[ i\mapsto A_i = \pmatriz{{cc} 2 & 0 \\ 0 & -2}, \qquad j\mapsto A_j = \pmatriz{{cc} 0 & 1 \\ -1 & 0}.\] of $G_2$. After suitable conjugation, we may assume that $G_2$ acts on $\FF_5^{2s}$ via $\psi \colon G_2 \to \GL_{2s}(5)$ such that $\psi(i)$ (respectively $\psi(j)$) is a block diagonal matrix $B_i$ (respectively $B_j$) with $s$ blocks $A_i$ (respectively $A_j$) on the diagonal. We claim $s = 1$.
If $s \geqslant 2$, then for the element $x = (1,1, 0, 1, 0, 0,..., 0, 0)^t \in \FF_5^{2s}$ one verifies that $\{\psi(q)x \colon q \in G_2\} \cap \langle x \rangle = \{x, -x\} \subsetneq \GEN{x} \setminus \{0\}$ and the element corresponding to $x$ in $\F(G)$ cannot be rational in $G$. Hence $s = 1$ and $\F(G) \simeq C_5 \times C_5$. As $G$ does not have elements of order $5 \cdot 7$, the action of $G_7$ on $\F(G) \setminus \lbrace 1 \rbrace$ has to be fixed-point free. Hence $\vert \F(G) \vert -1$ is a multiple of $\vert G_7 \vert = 7$. But this is in contradiction with $\F(G) \simeq C_5 \times C_5$.

\textbf{Case(2): $p=7$.}

	By \Cref{CyclicSylow} we have that $G_3 \simeq C_3$ and $G_5 \simeq C_5$. We set $H=G/\F(G)$. 
Recall that either $\ell_\F(G)=3$ or $\ell_\F(G)=4$ and $G$ has a normal section isomorphic to $2.S_4$. Therefore, we split the argument in this case according to these two possibilities.

\textbf{Subcase(2)(i):} $p=7$ and $\ell_\F(G)=3$.

As $G/\F_2(G)$ is nilpotent, we have that $5\nmid |G/\F_2(G)|$, and this implies that $H$ has a normal Sylow $5$-subgroup. Applying \Cref{CyclicSylow} to the group $H$ with $p=5$ and $q=2$, it follows that $G_2$ is either cyclic of order $4$ or isomorphic to $Q_8$. 
Since the action of $G_2$ on $G_{\{5,7\}}/G_7\simeq C_5$ must be faithful because $G$ has no elements of order $2\cdot 5$ we actually have $G_2\simeq C_4$. By \Cref{G2G3Exclusions}, $G$ does not have elements of order $12$, implying that $G_{\{2,3\}}\simeq C_3\rtimes C_4$. 
As $G$ contains an element of order $3\cdot 5$, we deduce that $G=C_7^d\rtimes (C_{15}\rtimes C_4)$ with $C_4$ acting by inversion on $G_3$ and faithfully on  $G_5$. 
In particular, $H$ has no non-trivial central elements and no normal $2$-subgroups. 
As $G$ does not have elements of order $3\cdot 7$ or $5\cdot 7$, the action of $H$ on $\F(G)$ is faithful. Thus, identifying $G_7$ with a $d$-dimensional vector space over $\FF_7$, this action is determined by an injective homomorphism $f:H\rightarrow \GL_d(\FF_7)$. {This implies that $5$ divides $7^d-1$ and hence $4\mid d$.} Let $t$ be an element of order $3$ in $G$. After suitable conjugation one may assume that $f(t)$ is a diagonal matrix with diagonal entries elements of $\FF_7$ of order $3$ (namely $2$ or $4$), but not all equal.  Now, since $G/\F(G)$ has a normal Sylow $3$-subgroup, $\F(G)G_3$ is the unique Hall ${\{3, 7\}}$-subgroup
	of $G$. Every $x\in G$ of order $7$ which is inverse semi-rational in $G$, already has to be inverse semi-rational in $\F(G)G_3$. However, if $x=(1,1,\dots,1)\in \FF_7^d=G_7$, then $\N_{G_{\{3,7\}}}(\GEN{x})=G_7$ and hence $B_{G_{\{3,7\}}}(x)=1$, a contradiction.

	\textbf{Subcase(2)(ii)}: $p=7$ and $\ell_\F(G)=4$.

In this case, $G$ has a normal section isomorphic to $2.S_4$ and hence so does $H$. As $\ell_\F(2.S_4)=3$ and $\F(2.S_4)\simeq Q_8$, the order of the Fitting subgroup $\F(H)$ is divisible by $2$. Then $5 \nmid |\F(H)|$, as $H$ has no element of order $2\cdot 5$. Also, as $H_3\simeq C_3$, we get that $\F(H)$ is a $3'$-group. Therefore, $F(H)$ is a $2$-group. Putting together, we have a group $H$ with the properties listed below:
	\begin{enumerate}
		\item $H$ is solvable and \cut.
		\item $H_3 \simeq C_3$ and  $H_5 \simeq C_5$.
		\item $H$ has a normal section $M/N$ isomorphic to $2.S_4$.
		\item $\ell_\F(H) = 3$. 
		\item $\F(H)$ is a $2$-group.
		\item $\GK(H) = (2-3-5).$
	\end{enumerate}
	We prove that there is no group satisfying conditions (1)-(6). For this, assume that $H$ is a group of minimal order with these properties and let $K$ be a minimal normal non-trivial subgroup of $H$. We  show that $H/K$ also satisfies conditions (1)-(6), contradicting the minimality of $H$. 
	
Clearly, $H/K$ satisfies (1). Also, being minimal normal subgroup of solvable group $H$, we have that $K$ is elementary abelian and hence $K\subseteq \F(H)$, so that $K$ is a $2$-subgroup. Consequently, $H/K$ satisfies (2).

 In order to prove that $H/K$ satisfies (3), we first verify that  either $M\cap K=1$ or $N\cap K\ne 1$. Indeed if $M\cap K\neq 1$ and  $N\cap K= 1$, then, it follows by minimality of $K$, that $K\subseteq M$. Moreover, under the natural projection $\eta \colon H \to H/N$,  $\eta(M) \simeq M/N $ and as $N\cap K=1$, it follows that $\eta(K)$ is an elementary abelian normal  subgroup of $\eta(M)$ isomorphic to $K$. Since the only elementary abelian normal subgroup of $2.S_4$ has order $2$, $K$ has order $2$ implying that $H$ has an element of order $2\cdot 5$, a contradiction. 
 
  So either $M\cap K=1$ or $N\cap K\ne 1$. Now, consider the natural projection $\kappa\colon H \to H/K$. In the first case $\kappa(M)/\kappa(N)  \simeq M/N \simeq 2.S_4$. In the second case, since  $K$ is a minimal normal subgroup of $H$, $N \slunlhd H$ and $N \cap K \not= 1$, necessarily $K \subseteq N \subseteq M$. Then, similar to the first case, we have $\kappa(M)/\kappa(N)  \simeq M/N \simeq 2.S_4$. This proves that $H/K$ satisfies (3).

Furthermore, as $H/K$ satisfies (3), we have that $3=\ell_\F(2.S_4)\leqslant \ell_\F(H/K)\leqslant \ell_\F(H)=3$, so that  $\ell_\F(H/K)=3$, i.e., $H/K$ satisfies (4). Since, $H/K$ satisfies (1)-(4), by the same arguments as for $H$ itself, it follows that $F(H/K)$ is a $2$-group, i.e., $H/K$ satifies (5). This further yields that $3-5$ is an edge of ${\GK(}H/K{)}$ and as $2-3$ is an edge of $\GK(2.S_4)$, it is also an edge of $\GK(H/K)$, thereby implying that $H/K$ satisfies (6).
\end{proof}		
We have thus proved that the GK-graph of a solvable \cut group with $4$ vertices has at least $4$ edges. We next proceed to eliminate impossible choices of such graphs.

\subsection{Eliminating  non-feasible graphs with more than 3 edges}

 We now complete the proof of  \Cref{T4}.\\
 
	\noindent\textbf{\textit{Proof of \Cref{T4}.}} As already mentioned the first part of \Cref{T4} follows from \Cref{PGroups}. For the second part, let $\Gamma=\GK(G)$ for a finite solvable \cut group $G$, such that $\Gamma$ has $4$ vertices. 

 By \Cref{4Vertices3Edges}, $\Gamma$ has at least 4 edges. We want to prove that it is one of the seven graphs displayed in \Cref{Four}. As the graph with six edges appears in \Cref{Four}, we may assume that $\Gamma$ has $4$ or $5$ edges. There are $6$ possible graphs with $5$ edges but only the $3$ displayed in \Cref{Four} satisfy the conditions of  \Cref{6and35}. Now, if $\Gamma$ has $4$ edges, there are $15$ possible choices, out of which only $5$  satisfy the conditions of \Cref{6and35}, while $3$ of them appear in \Cref{Four}. 
The remaining two are  $\Gamma_5$ and $\Gamma_6$ as given in  \Cref{Fig2Graphs}.

\begin{figure}[ht!]
	\begin{tabular}{ccc}
	\begin{subfigure}{.15\textwidth}
		\begin{tikzpicture}
		\node[label=west:{$2$}] at (0,1.5) (2){};
		\node[label=east:{$3$}] at (1.5, 1.5) (3){};
		\node[label=west:{$5$}] at (0, 0) (5){};
		\node[label=east:{$7$}] at (1.5, 0) (7){};
		\foreach \p in {2,3,5,7}{
			\draw[fill=black] (\p) circle (.075cm);
		}
		\draw (2)--(3);
		\draw (2)--(5);
		\draw (2)--(7);
		\draw (3)--(7);
		\end{tikzpicture}
	\end{subfigure} 
	
		& \hspace{1cm} & 
		
			\begin{subfigure}{.15\textwidth}
			\begin{tikzpicture}
			\node[label=west:{$2$}] at (0,1.5) (2){};
			\node[label=east:{$3$}] at (1.5, 1.5) (3){};
			\node[label=west:{$5$}] at (0, 0) (5){};
			\node[label=east:{$7$}] at (1.5, 0) (7){};
			\foreach \p in {2,3,5,7}{
				\draw[fill=black] (\p) circle (.075cm);
			}
			\draw (2)--(3);
			\draw (2)--(5);
			\draw (3)--(5);
			\draw (3)--(7);
			\end{tikzpicture}
		\end{subfigure}
		\\
		$\Gamma_5$	& & $\Gamma_6$ 
	\end{tabular}
\caption{}\label{Fig2Graphs}\end{figure}

So to finish the proof of \Cref{T4}, it remains to prove that if $\Gamma$ is the GK-graph of a finite solvable \cut group then $\Gamma\ne \Gamma_5$ and $\Gamma\ne \Gamma_6$. Let $G$ be a finite solvable \cut group of minimal order such that $\Gamma=\GK(G)$ is $\Gamma_5$ or $\Gamma_6$. Then, using \Cref{MANS} and arguing as in the proof of \Cref{4Vertices3Edges}, we have that $\F(G)$ is a Sylow $p$-subgroup of $G$, which is elementary abelian, with $p=5$ or $7$. 

\textbf{Case(1):} $\Gamma=\Gamma_5$ and $p=5$.

By \Cref{CyclicSylow}, $G_3 \simeq C_3$, and thus by \Cref{G2G3Exclusions}, $G$ cannot have an element of order $3\cdot 7$, a contradiction.

\textbf{Case(2):} $\Gamma=\Gamma_5$ and $p=7$.

By \Cref{CyclicSylow}, $G_5$ is cyclic of order $5$ generated by $x$, say.
Let $H$ be a Hall $\{3,5\}$-subgroup containing $x$. 
	By \Cref{PrimePower}, $x$ is rational and hence $|B_G(x)|=4$. Moreover, as $G$ has no elements of order $3\cdot 5$, we have that $\C_G(x)$ is a $3'$-group, and {therefore} so is $\N_G(\GEN{x})$. Thus $\F(H)$ is a $5'$-group. Hence, $\F(H)$ is a $3$-group and the second fitting layer of $H$, i.e., $\F_2(H)/\F(H)$, is a $5$-group, and in particular isomorphic to $C_5$. Again, because $\N_G(G_5)$ is a $3'$ group, $\ell_\F(H)=2$, therefore, $H\simeq G_3\rtimes C_5$. Now, let $K=\GEN{H,G_7}$, a Hall $\{3,5,7\}$-subgroup of $G$ containing $H$.  Then, $N=\GEN{G_3,G_7}$ is a Hall $\{3,7\}$-subgroup of $G$, which is normal in $K$, and $\C_K(n)\subseteq N$ for every $n\in N\setminus \{1\}$, because $G$ has no elements of order $3\cdot 5$ or $5\cdot 7$. Thus, $K$ is a Frobenius group with kernel $N$. In view of  \cite[Theorem 11.4.3]{JdR1}, $N$ must be nilpotent, implying that the elements of order $7$ are not inverse semi-rational in $N$ \cite[Theorem 1]{BMP17}, and therefore also not inverse semi-rational in $G$, a contradiction.

\textbf{Case(3):} $\Gamma=\Gamma_6$ and $p=5$.

By \Cref{CyclicSylow}, $G_7\simeq C_7 =\GEN{x}$, say. As $x$ is inverse semi-rational,  $B_G(x)$ is of order $3$ or $6$. 
As there is no element of order $2\cdot 7$ in $G$, $\C_G(x)$ does not have an element of order $2$. Hence, $4\nmid |\N_{G}(\GEN{x})|$.

Let $H$ be a Hall ${\{2,7\}}$-subgroup containing $x$.
Since the elements of order $5$ are inverse semi-rational in $G$, $H$ should have elements of order $4$ and this along with the fact that $G$ does not have elements of order $2\cdot 7$, yields that $F(H)$ is a $2$-subgroup. Also, $\F_2(H)/\F(H)$ has order $7$ and index at most $2$ in $H/\F(H)$, as $4\nmid |\N_{G}(\GEN{x})|$. Therefore $\F(H)$ has index at most $2$ in a Sylow $2$-subgroup of $G$ and $H=\F(H)\rtimes (\GEN{x}\rtimes C_2)$ or $H = \F(H)\rtimes\GEN{x}$.

Consider $K=\langle G_5,\F(H),x\rangle$ and $N=\langle G_5,\F(H)\rangle$. Clearly $K=N\rtimes \GEN{x}$ and, as $G$ has elements of order {neither}  $2\cdot 7$ nor $5\cdot 7$, it follows that $\C_K(n)\subseteq N$ for every $n \in N \setminus \{1\}$. This proves that $K$ is a Frobenius group with kernel $N$, so that $N$ must be nilpotent. However if $g\in G$ has order $5$ then $g^h=g^2$ for some $h\in G$, since $g$ is rational in $G$. Then $g,h^2\in N$ and $[g,h^2]=g^{-1}g^{h^2}=g^{-2}\ne 1$, contradicting the nilpotency of $N$. 

\textbf{Case(4):} $\Gamma=\Gamma_6$ and $p=7$.

By \Cref{CyclicSylow},  $G_5 \simeq C_5$ and $G_2$ is either $C_4$ or $Q_8$, but then by \Cref{G2G3Exclusions}, $G$ has no elements of order $2\cdot 5$, a contradiction.  \hfill  $\Box$\\

\section{Applications: Frobenius, $2$-Frobenius, supersolvable and metanilpotent \cut groups}\label{SectionApplications}

In this section we show how one can apply the previous results and techniques to obtain a complete description of the GK-graphs 
	 realized by interesting {sub}classes of solvable \cut groups.
	 In particular, we demonstrate that the GK-graphs that can appear are the same for classes of groups that are notably different.

We begin with the following observation.

	\begin{proposition}\label{Nilpotent} 	The following are equivalent for a graph $\Gamma$.	
	\begin{enumerate}
		\item $\Gamma=\GK(G)$ for some finite non-trivial cyclic \cut group $G$. 			
		\item $\Gamma=\GK(G)$ for some finite non-trivial abelian \cut group $G$. 			
		\item $\Gamma=\GK(G)$ for some finite non-trivial nilpotent \cut group $G$. 			
		\item $\Gamma$ is one of the graphs (a), (b) or (d) in \Cref{Less4}.
	\end{enumerate}
Furthermore, $\Gamma$ is the {GK-}graph of a finite non-trivial nilpotent rational group if and only if $\Gamma$ is graph~$(a)$.
\end{proposition}

\subsection*{Frobenius and 2-Frobenius \cut groups.} It is known that 2-Frobenius groups are solvable and Frobenius \cut groups are solvable (see \Cref{GKLNS} and \cite[Proposition~4.1]{Bac}).
By \Cref{ConnectedComponents}, the GK-graph of a solvable group is disconnected if and only if the group is a Frobenius group or a $2$-Frobenius group. The 5 disconnected graphs in Theorems~\ref{TLess4} and \ref{T4}, namely (c), (e), (g), (h) and (l) are realized by Frobenius \cut groups in \Cref{PGroups}. Similarly, the two disconnected GK-graphs in \Cref{PrimeGraphRat} can be realized by rational Frobenius groups. Consequently, we have the following:
\begin{proposition}\label{Frobeniuscut} The following are equivalent for a graph $\Gamma$.
		\begin{enumerate}
			\item $\Gamma=\GK(G)$ for some finite non-trivial Frobenius \cut group $G$. 			
			\item $\Gamma$ is one of the graphs (c), (e), (g), (h) or (l) in \Cref{Less4}.
		\end{enumerate}
		Furthermore, $\Gamma$ is the graph of a finite non-trivial Frobenius rational group if and only if $\Gamma=(c)$ or $(e)$.
	\end{proposition}

We now classify the GK-graphs realized by a $2$-Frobenius \cut group.

\begin{proposition}\label{2Frobeniuscut} The following are equivalent for a graph $\Gamma$.
	\begin{enumerate}
		\item $\Gamma=\GK(G)$ for some finite non-trivial $2$-Frobenius \cut group $G$. 			
		\item $\Gamma$ is one of the graphs (c), (e), (g) or (l) in \Cref{Less4}.
	\end{enumerate}
Furthermore, $\Gamma$ is the graph of a finite non-trivial $2$-Frobenius rational group if and only if $\Gamma$  is graph $(c)$.
	 \end{proposition}
\begin{proof} By Theorems~\ref{TLess4} and \ref{T4}, to prove the first statement it suffices to prove that the graph (h) cannot be realized as the GK-graph of a $2$-Frobenius \cut group and  that each of the graphs (c), (e), (g) and (l) is the GK-graph of a $2$-Frobenius \cut group. 
	
We begin by constructing a 2-Frobenius \cut group that realizes the graph (l) = (2-3   7). For this, observe that the Frobenius \cut group $H = C_7 \rtimes_\Frob C_6$ can be realized as a permutation group \linebreak $\langle (1,2,3,4,5,6,7), (1,3,2,6,4,5)\rangle$. Writing this permutation group as $7 \times 7$-permutation matrices and reading modulo the obvious fixed subspace of all elements with coordinate sum $0$ gives this group  as 
group of $6 \times 6$-matrices 
\[\left\langle A = \begin{pmatrix}
 . & . & .& . & . & -1 \\
 1 & . & .& . & . & -1 \\
 . & 1 & .& . & . & -1 \\
 . & . & 1& . & . & -1 \\
 . & . & .& 1 & . & -1 \\
 . & . & .& . & 1 & -1 
 \end{pmatrix}, ~
B = \begin{pmatrix}
 . & . & .& . & 1 & . \\
 . & . & 1& . & . & . \\
 1 & . & .& . & . & . \\
 . & . & .& . & . & 1 \\
 . & . & .& 1 & . & . \\
 . & 1 & .& . & . & . 
 \end{pmatrix}
 \right\rangle \]
(dots indicate zeros). Reducing these matrices modulo $2$ gives a $6$-dimensional faithful representation of $H$ over $\mathbb{F}_2$. 

We can hence form the semi-direct product $G = C_2^6 \rtimes (C_7 \rtimes_\Frob C_6)$, where we identify $C_2^6$ as a $6$-dimensional $\FF_2$-vector space and $C_7 \rtimes_\Frob C_6$ acts by the above representation. Observe that $A$ is the companion matrix of the $7$th cyclotomic polynomial. 
Hence $\F_2(G) = C_2^6 \rtimes_\Frob C_7$ is a Frobenius group and $G$ is a $2$-Frobenius group with the desired GK-graph (2-3   7).  

To verify that $G$ is \cut observe first that every element of $G$ of order different from $7$ has order dividing $12$.
Let $a$ and $b$ be generators of $H$ of order $7$ and $6$, respectively. It is clear that $a$ is rational in $H$, so to prove that $G$ is \cut we only have to show that its elements of order $12$ are inverse semi-rational in $G$. 
Let $g$ be an element of order $12$ in $G$ and consider the Hall $7'$-subgroup $K = N\rtimes C_6$, where $N = \F(G) \simeq C_2^6$.
By Hall's Theorem, $g$ is conjugate to an element in $K$ and therefore we may assume that $g\in K$.
The center of $K$ is generated by the element $f$ of $N$ without any trivial entry. 
Let $X$ be the subset of $N$ formed by the elements with an odd number of trivial entries. 
Then $g=xb$ or $g=xb^{-1}$ with $x\in X$ and we may assume without loss of generality that $g=xb$. 
We claim that $\C_K(g)=\GEN{g}$. Indeed, let $h\in \C_K(g)$ and write $h=yb^i$ with $y\in C_2^6$ and $0\leqslant i\leqslant 5$. 
Then $g^ih^{-1}\in \C_K(g)\cap N=\Z(K)=\GEN{f}$,
so that $h$ is either $g^i$ or $fg^i=g^{6+i}$. This finishes the proof of the claim.
Hence the conjugacy class of $g$ in $K$ has cardinality $|K|/12 = 32$. 
The set $Xb=\{yb:y\in X\}$ is closed under conjugation  in $K$, so it contains the conjugacy class of $g$ in $H$.
Moreover $|Xb|=32$  and hence $Xb$ is exactly the conjugacy class of $g$ in $K$. 
As $g^7=fg\in Xb$, it follows that $g$ and $g^7$ are conjugate in $K$ and hence $g$ is inverse semi-rational both in $K$ and $G$, as desired.

In the same spirit one can define  $2$-Frobenius \cut groups $C_2^2 \rtimes (C_3 \rtimes_\Frob C_2), C_2^4 \rtimes (C_5 \rtimes_\Frob C_4)$ and $C_3^6 \rtimes (C_7 \rtimes_\Frob C_3)$ realizing the graphs (c) = (2   3), (e) = (2   5) and (g) = (3  7), respectively, using the following matrix realizations of the upper Frobenius groups
\[
 \left\langle C = \begin{pmatrix}
 . & -1 \\
 1 & -1
 \end{pmatrix}, D =
\begin{pmatrix}
 . & 1 \\
 1 & .
 \end{pmatrix}
 \right\rangle, \qquad 
 \left\langle E = \begin{pmatrix}
 .& . & . & -1 \\
 1& . & . & -1 \\
 .& 1 & . & -1 \\
. & . & 1 & -1 
 \end{pmatrix}, 
F = \begin{pmatrix}
 .& . & 1 & . \\
 1& . & . & . \\
 .& . & . & 1 \\
 .& 1 & . & . 
 \end{pmatrix}
 \right\rangle, \qquad
  \left\langle A, B^2 \right\rangle,
\] respectively. 
Note that the group $C_2^2 \rtimes (C_3 \rtimes_\Frob C_2)$ constructed is isomorphic to the symmetric group $S_4$.

It remains to prove that there is no $2$-Frobenius \cut group realizing the GK-graph (h)=(2-3   5).  Assume that $G$ is a $2$-Frobenius \cut group realizing that graph. Then, by \Cref{GKLNS} and \Cref{FrobeniusCut}, $G/\F(G) \simeq C_5 \rtimes_\Frob C_4$. Replacing $G$ by $G/\langle \O_2(G), \Phi(\O_3(G)) \rangle$ (where $\Phi$ denotes the Frattini subgroup), if necessary, we may assume that $\F(G)$ is an elementary abelian $3$-group. Since $\F(G)$ and $G/\F(G)$ are coprime, $\F(G)$ considered as an $\FF_3[C_5 \rtimes_\Frob C_4]$-module is semi-simple and faithful. We may assume that it is a single copy of the unique faithful $\FF_3[C_5 \rtimes_\Frob C_4]$-module (where the actions of elements $e$ and $f$ of order $5$ and $4$, respectively, are defined by the matrices $E$ and $F$ above). Now $v = (1,1,1,1) \in \FF_3^4$ is an element of order $3$ and commutes with the element $f$ and hence $x = vf$ has order $12$. Since $G_2 \simeq C_4$, it follows from \Cref{G2G3Exclusions}\eqref{Not4pG2Cyclic} that $x$ is not inverse semi-rational in $G$ and $G$ is not a \cut group.

Furthermore, since the graph $(2  ~~5)$ cannot be the GK-graph of a rational $2$-Frobenius group, by \cite[Lemma 4]{DIM}, $ (2  ~~ 3)$ is the only graph that can be realized by a rational $2$-Frobenius group. \end{proof}

\subsection*{Supersolvable groups} 
Note that supersolvable groups are nilpotent-by-abelian \cite[5.4.10]{Robinson}. We begin by restricting the set of possible vertices
of the GK-graph of finite nilpotent-by-abelian groups and we will see 
that the sets of vertices are the same for supersolvable groups.

\begin{lemma}\label{lem:primespec_supersolvable} Let $G$ be a finite nilpotent-by-abelian group.
 \begin{enumerate}
  \item If $G$ is rational, then $\pi(G) \subseteq \{2,3\}$.
  \item If $G$ is \cut, then $\pi(G) \subseteq \{2,3,5\}$ or $\pi(G)  \subseteq \{2,3,7\}$.
 \end{enumerate}
\end{lemma}

\begin{proof}  (1) Suppose first that $G$ is rational.
By \cite{Gow}, $\pi(G) \subseteq \{2,3,5\}$.  Assume that $5$ divides $|G|$. Since $G/G'$ is an abelian rational group it is an elementary abelian $2$-group. Hence $5$ divides $|G'|$. Pick $x \in \Z(G')$ of order $5$. Then $B_G(x)$ is cyclic of order $4$, but this is impossible since the exponent of $G/G'$ divides $2$.

(2)  Suppose now that $G$ is \cut.  Then by \Cref{2357}, $\pi(G) \subseteq \{2,3,5,7\}$. Assume that $35$ divides $|G|$. Since $G/G'$ is an abelian \cut group, it has exponent divisible by $4$ or $6$. Hence $35$ divides $|G'|$. Pick $x \in \Z(G')$ of order $5\cdot 7$. Then $B_G(x)$ contains an element of order $12$ and this gives a contradiction.\end{proof}

\begin{proposition}\label{RationalSupersolvable}
	The following are equivalent for a graph $\Gamma$.
	\begin{enumerate}
		\item $\Gamma=\GK(G)$ for some finite non-trivial metacyclic rational group $G$. 			
		\item $\Gamma=\GK(G)$ for some finite non-trivial metabelian rational group $G$. 			
		\item $\Gamma=\GK(G)$ for some finite non-trivial supersolvable rational group $G$. 		
		\item $\Gamma=\GK(G)$ for some finite non-trivial nilpotent-by-abelian rational group $G$. 					
		\item $\Gamma$ is one of the graphs (a), (c) or (d) in \Cref{Less4}.
	\end{enumerate}
\end{proposition}

\begin{proof}
(1) implies (2), (1) implies (3) and (2) implies (4) are clear. 
(3) implies (4) is a well known result \cite[5.4.10]{Robinson}.
(5) implies (1) is a consequence of \Cref{PGroups}. 

Hence, it suffices to prove that (4) implies (5). For this, suppose that $G$ is a non-trivial nilpotent-by-abelian rational group and $\Gamma=\GK(G)$. 
By \Cref{lem:primespec_supersolvable}, $\pi(G)\subseteq \{2,3\}$. 
As $G$ is rational, $2$ is one of the vertices of $\Gamma$ and hence $\Gamma$ is one of the graphs (a), (c) or (d), by \Cref{TLess4}. 
\end{proof}

\begin{proposition}\label{cutSupersolvable}
	The following are equivalent for a graph $\Gamma$.
\begin{enumerate}	
	\item $\Gamma=\GK(G)$ for some finite non-trivial metacyclic \cut group $G$. 				
	\item $\Gamma=\GK(G)$ for some finite non-trivial metabelian \cut group $G$. 			
	\item $\Gamma=\GK(G)$ for some finite non-trivial supersolvable \cut group $G$. 		
	\item $\Gamma=\GK(G)$ for some finite non-trivial nilpotent-by-abelian \cut group $G$. 					
	\item $\Gamma$ is one of the graphs  (a) -- (g) or (j) -- (o) in \Cref{Less4}.
\end{enumerate}
\end{proposition}

\begin{proof}
Arguing as in the proof of \Cref{RationalSupersolvable}, it is enough to prove (4) implies (5) and (5) implies (1).

We first show that (4) implies (5). 
 Suppose that $G$ is nilpotent-by-abelian and $\Gamma=\GK(G)$. 
By \Cref{lem:primespec_supersolvable} we have that $\pi(G)\subseteq \{2,3,5\}$ or $\pi(G)\subseteq \{2,3,7\}$. Hence, by \Cref{TLess4}, necessarily $\Gamma$ is one of the graphs (a)-(o) in \Cref{Less4}. 
  It is thus required to prove that $\Gamma$ is neither the graph (h) nor the graph (i) in \Cref{Less4}. To this end, we show that if $3$ and $5$ divide $|G|$, then $G$ necessarily has an element of order $3\cdot 5$.
	Since $G/\F(G)$ is an abelian \cut group, $5$ does not divide $|G/\F(G)|$. 
	 Therefore if $5$ divides the order of $G$ then $\F(G)$ has a central element of order $5$, yielding an element of order $4$ in $G/\F(G)$ and consequently  $3$ does not divide $|G/\F(G)|$. This implies that if $15$ divides $|G|$ then 15 divides $|\F(G)|$ and hence $G$ has an element of order $3\cdot 5$.

We next show that (5) implies (1). 
Note that the groups corresponding to (a) -- (d), (g)  and (l) -- (o) in \Cref{fig:Examples} are metacyclic, and \cut. By \Cref{FrobeniusCut}, there is a unique Frobenius \cut group $F$ of the form $C_5 \rtimes_\Frob C_4$, which is clearly metacyclic and $\GK(F)$ is the graph (e). 
Moreover $\GK(F\times C_2)$ is the graph (f).
Consider now the group $G=\GEN{x}\rtimes \GEN{y}$ with $|x|=15$, $|y|=4$ and $x^y=x^2$. Then $G$ and $G\times C_2$ are metacyclic {\cut groups 
	 whose Gruenberg-Kegel graphs are} the graphs  (j) and (k) respectively.  
\end{proof}

\subsection*{Metanilpotent groups}Recall that a group $G$ is called \emph{metanilpotent} if it contains a nilpotent normal subgroup $N$ such that $G/N$ is nilpotent or, equivalently, if $\ell_\F(G) \leqslant 2$.

\begin{proposition} 
\label{MetanilpotentRational}
The following are equivalent for a graph $\Gamma$.
\begin{enumerate}
	\item $\Gamma=\GK(G)$ for some finite non-trivial abelian-by-nilpotent rational group $G$.  				
	\item $\Gamma=\GK(G)$ for some finite non-trivial metanilpotent rational group $G$.	
	\item $\Gamma$ is one of the graphs (a), (c)-(f) or (k) in \Cref{Less4}.
\end{enumerate}
\end{proposition}

\begin{proof} (1) implies (2) is clear. (3) implies (1) follows from the fact that the groups in (a), (c)-(f) and (k) in \Cref{fig:Examples} are abelian-by-nilpotent and rational.
Finally, we prove (2) implies (3). By \Cref{PrimeGraphRat}, it suffices to show that if $G$ is a non-trivial metanilpotent rational group then $\GK(G) \ne (3-2-5)$. 
Assume the contrary. Then $G/\F(G)$ is a rational nilpotent group and hence  a $2$-group. This implies that $\F(G)$ is a nilpotent group of order divisible by $3$ and $5$ and hence $\GK(G)$ contains the edge $3-5$, a contradiction.
\end{proof}

\begin{proposition}
The following are equivalent for a graph $\Gamma$.
\begin{enumerate}
	\item $\Gamma=\GK(G)$ for some finite non-trivial abelian-by-nilpotent \cut group $G$. 					
	\item $\Gamma=\GK(G)$ for some finite non-trivial metanilpotent \cut group $G$. 		
	\item $\Gamma$ is one of the graphs (a)-(g) or (j)-(r) in Figures \ref{Less4} and \ref{Four}. 
\end{enumerate}
\end{proposition}

\begin{proof}
Clearly, (1) implies (2). Also, observe that the only groups in \Cref{fig:Examples} which are not abelian-by-nilpotent are  precisely those realizing the graphs (h) and (i). This proves (3) implies (1). So, we prove (2) implies (3).
In view of Theorems~\ref{TLess4} and \ref{T4}, it suffices to prove that none of the graphs (h)-(i)  and (s)-(v) are realizable as the GK-graph of a metanilpotent group. 

Let $G$ be a metanilpotent \cut group and let $\Gamma=\GK(G)$. 
We first observe that $\Gamma$ is not the graph (h).
This follows at once because, as (h) is not connected, by \Cref{ConnectedComponents}, $G$ must be either Frobenius or 2-Frobenius. The former is not compatible with \Cref{FrobeniusCut} and the latter is in contradiction with \Cref{2Frobeniuscut}. 

Since $G/\F(G)$ is nilpotent and \cut, neither $5$ nor $7$ divides $[G:\F(G)]$.
Thus $G_5$ and $G_7$ are normal in $G$. 
Hence, if $|G|$ is divisible by $5$ and $7$ then $G$ has an element of order $5\cdot 7$. 
This implies that $\Gamma$ is none of (s), (t) or (v). 
It remains to prove that $\Gamma$ is neither (i) nor (u). 

Suppose that $\Gamma$ is the graph (u). 
Then $\F(G)$ has an element of order $5\cdot 7$ and, as this element is inverse semi-rational, $G/\F(G)$ has elements of order $4$ and $3$.
Since $G/\F(G)$ is nilpotent and \cut necessarily the Sylow 2-subgroup of $G$ is non-abelian.
As $G_5$ is normal in $G$ and $G$ does not have elements of order $2\cdot 5$, by \Cref{CyclicSylow}, the Sylow 2-subgroup of $G$ is a quaternion group of order $8$. 
In particular, $G$ does not have an abelian section of order $8$. 
As $B_G(g)$ is abelian for every $g\in G$, it follows that $|B_G(g)_2|$ divides $4$. 
Since $g$ is inverse semi-rational then $|g|\ne 3\cdot 5 \cdot 7$.
Thus $G$ does not have elements of order $3\cdot 5\cdot 7$ while it has elements of order $5\cdot 7$ and elements of order $3\cdot 7$.
This implies that $\F(G)=G_5\times G_7$ and $G/F(G)\simeq G_2\times G_3$. 
Let $H$ be a Sylow $2$-subgroup of $G$. Then $G_5H$ is a Frobenius group with Frobenius kernel $G_5$. 
The group $G/F(G)$ is \cut and isomorphic to $H\times C_3$ with $H$ a $2$-group.  Suppose that $H$ is not rational. Then there is $h\in H$ and and odd integer $r$ such that $h^r$ is not conjugate to $h$ in $H$. 
	Let $g$ be an element of order $3$ in $G/F(G)$. 
	Let $s$ be an integer such that $s\equiv r\mod |h|$ and $s\equiv 1 \mod 3$.
	Then $gh=hg$, $g^s=g$ is not conjugate to $g^{-1}$ in $G/F(G)$ and $h^s=h^r$ is not conjugate to $h$ in $G/F(G)$. Therefore $(gh)^s$ is conjugate to neither $gh$ nor  $(gh)^{-1}$ in $G/F(G)$, in contradiction with the fact that $G/F(G)$ is \cut.  So $H$ is rational.
Since $\F(G)H$ contains a Sylow $2$-subgroup of $G$ and is normal in $G$, it contains all Sylow $2$-subgroups, and hence every $5$-element of $G$ is rational in $G_5H$. Thus $G_5H$ is a Frobenius rational group. 
Let $g$ and $h$ be commuting elements of $G$ of order $3$ and $7$ respectively, so that $g$ does not commute with any $5$-element, as $G$ has no elements of order $3\cdot 5\cdot 7$.
Then $K=\GEN{G_5H,g}$ is a \cut metanilpotent group with $\GK(K)$ equal to the graph (h). This contradicts the previously proven fact that (h) is not the GK-graph of a cut metanilpotent group.

Finally, suppose that $\Gamma$ is the graph (i). 
Then the Sylow subgroup of $G$ is cyclic of order $3$, by \Cref{CyclicSylow}, and $G/\F(G)$ has an element of order $12$. Thus, as in the previous paragraph, the Sylow 2-subgroup of $G$ is non-abelian.
We may assume that $G$ has minimal order with $\Gamma(G)=\Gamma$. 
Let $A$ be a minimal normal subgroup of $G$ contained in $G_5$.
If $A\ne G_5$ then $\pi(G/A)=\{2,3,5\}$, and the minimality of $\Gamma$ implies that $\GK(G/A)$ has less edges than $\Gamma$. By \Cref{TLess4}, necessarily $\GK(G/A)$ is the graph (h), a contradiction.
Thus $A=G_5$ and $G_5$ is elementary abelian. 
Using again the minimality of the action of $G$ it follows that the action of $G_{5'}$ on $G_5$ is faithful and irreducible.
In particular, $\F(G)=G_5$ and therefore $G=G_5\rtimes H$ with $H=H_2\times H_3$ for $H_2$ a non-abelian $2$-group and $H_3$ a cyclic group of order $3$. 
We consider $G_5$ as an irreducible right $\FF_5 H$-module via the multiplication $x\cdot h=x^h$ for $x\in G_5$ and $h\in H$.
This module is the tensor product of an irreducible $\FF_5 H_2$-module $M$ and an irreducible $\FF_5 H_3$-module $N$, i.e. if $m\in M$, $n\in N$ and $h\in H$ then $(m\otimes n)h=mh_2\otimes nh_3$. 
As $G$ does not have elements of order $3\cdot 5$ and $H_3\simeq C_3$, we have that $N$ has degree $2$.
As every element of $G_5$ is rational in $G$, for every $x\in G_5$ there is $g\in H_2$ such that $x\cdot g=2x$.
We claim that there is no $g\in H$ such that $m\cdot g=2m$ for every $m\in M$.
For if $g\in H$ satisfies $m\cdot g=2m$ for every $m\in M$ then $x\cdot (gh)=(2x)\cdot h=x\cdot (hg)$ for every $h\in H_2$ and every $x\in G_5$. 
Then $[g,h]=1$ because the action of $H$ on $G_5$ is faithful. 
Thus $g$ is a  central element of order multiple of $4$ in $H_2$ and hence $G/\F(G)\simeq H$ has a central element of order $12$, which is not possible. 
This proves the claim. 
The claim implies that $M$ has two linearly independent elements $x_1$ and $x_2$ such that there is no $g\in H_2$ such that $x_i\cdot g=2x_i$ for $i=1,2$.
Let $y_1,y_2\in N$ be linearly independent and take $a=x_1\otimes y_1+x_2\otimes y_2$. 
As $a$ is inverse semi-rational there is $g\in H_2$ such that $a\cdot g=2a$. 
Then $x_i\cdot g=2x_i$ for $i=1,2$, a contradiction.
\end{proof}

 \section{The Prime Graph Question for \cut groups}\label{SectionPQ}

In this section, we prove \Cref{PrimeGraphCut}, which almost answers the Prime Graph Question \PQ for \cut groups and implies a positive answer for rational groups. We first recall the reduction result of Kimmerle and A.~Konovalov to almost simple groups.
Recall that a group $A$ is \emph{almost simple}, if there exists a non-abelian simple group $S$ such that $\Inn(S) \leqslant A \leqslant \Aut(S)$. Since $S \simeq \Inn(S)$ this means that $A$ is sandwiched between a non-abelian-simple group and its automorphism group; the group $S$ is called the \emph{socle of $A$}. Note that each non-abelian simple group $S$ defines its own family of almost simple groups, parametrized by the conjugacy classes of subgroups of $\Out(S) = \Aut(S)/\Inn(S)$.

\begin{theorem}[{\cite[Theorem 1.2]{KimmerleKonovalov2016}}]\label{reduction}
	Let $G$ be a finite group. The Prime Graph Question \PQ has a positive answer for $G$ if it has a positive answer for all almost simple images of $G$.
\end{theorem}

Using the  reduction result in \Cref{reduction}, \PQ was answered affirmatively for several classes of groups including many non-solvable groups, e.g.\ for all groups $G$ with $|\pi(G)| \leqslant 3$ \cite{KKStAndrews,BM17}.
Recently, Trefethen  \cite{Tre19} gave a list of all non-abelian composition factors of \cut\ groups (note that it is still an open problem to determine the abelian composition factors of \cut\ groups). This list together with the so-called HeLP method enables us to prove \Cref{PrimeGraphCut}.\\

\noindent\textbf{\textit{Proof of \Cref{PrimeGraphCut}.}}\ Since images of \cut\ groups are \cut\ by \Cref{ElementaryDP}, it suffices to give a positive answer to \PQ for all almost simple \cut groups that do not map onto the monster $M$ by the above mentioned reduction \Cref{reduction}. The socle of every almost simple \cut\ group is in the list of non-abelian composition factors of \cut\ groups in \cite[Theorem 1.1]{Tre19}. 
\PQ has a positive answer for all almost simple groups with an alternating socle \cite{BMAn}.
Furthermore, in \cite[Corollary~1.2]{CM21} an affirmative answer to (PQ) is given for all almost simple groups with sporadic socle other than the O'Nan simple group $O'N$ or the simple monster group $M$. However the two almost simple groups $G$ 
 with socle $O'N$ are not \cut because both have an element $g$ of order $19$ and no element of order $9$, see e.g.\ the ATLAS \cite{ATLAS} or its online version \cite{OnlineATLAS} (\url{http://brauer.maths.qmul.ac.uk/Atlas/v3/spor/ON/}) 
and hence $g$ is not inverse semirational in $G$, by \Cref{PrimePower}.

It remains to consider almost simple groups with a \cut simple composition factor of Lie type and for that we go through the list  of the $16$ composition factors of Lie type in \cite{Tre19}. 
As a matter of fact (PQ) has been proven for all of them.
 In \Cref{Table_PQcut},  
each column corresponds to such a composition factor. The first row contains the name of the simple group $S$ and the second row gives the structure of the outer automorphism group $\Out(S)$ of $S$ listed from \url{http://brauer.maths.qmul.ac.uk/Atlas/v3}. The ATLAS notation is used throughout. For example, $S.2_3$ denotes the extension of the non-abelian simple group $S$ which corresponds to the $3$rd conjugacy class of subgroups of $\Out(S)$ of order $2$.
The last row contains  references where the Prime Graph Question was answered affirmatively for all groups in the respective column.
\hfill $\Box$


As the monster group is not rational \cite[pp.~220-234]{ATLAS}, \Cref{PrimeGraphRational} is a consequence of \Cref{PrimeGraphCut}.

In our final remark we list all the almost simple \cut (respectively, rational) groups.

\begin{remark}\label{ASCut}
The third and fourth rows of \Cref{Table_PQcut} contain all the almost simple groups with socle $S$, arranged according to whether they  are \cut or not. To check the latter, probably the most convenient way is to use the known character tables of these groups, see e.g.\ \cite{GAP4}, and the function \verb+IsCutGroup+ included in Section~\ref{SubsectionRationalCut}.
For example, calling \texttt{IsCutGroup(CharacterTable("L2(7)"))} one verifies that $L_2(7)$ is \cut, while that $L_2(7).2$ is not \cut can be checked by calling \texttt{IsCutGroup(CharacterTable("L2(7).2"))}.	

\Cref{Table_PQCut2} displays the remaining almost simple \cut groups arranged by the type of socle.
We sketch how to verify this. It is well-known that symmetric groups are rational and by \cite[Theorem 4.6]{Fer04} an alternating group $A_n$ with $n\geqslant 5$ is \cut if and only if $n\in \{7, 8, 9, 12\}$. 
As $|\Out(A_n)|=2$ for $n\ne 6$ and $\Out(A_6)\simeq C_2\times C_2$, to decide which almost simple groups with alternating socle are \cut we are left with the exceptional almost simple groups $A_6.2_2 = \PGL(2,9)$, $A_6.2_3 = M_{10}$ and $\Aut(A_6)=A_6.2^2$. 
Now calling $\texttt{IsCutGroup(CharacterTable("A6.2\_2"))}$ in \textsf{GAP} one verifies that $A_6.2_2$ is not \cut. Similarly one verifies that both $A_6.2_3$ and $A_6.2^2$ are \cut. Using the list of possible sporadic socles from \cite[Theorem 1.1]{Tre19}, one can proceed similarly for the remaining cases.

Likewise, using \cite[Theorem B]{FS89}, one obtains that the almost simple rational groups are: \[S_n  ( n\geqslant 5 ),\  A_6.2^2 ,\  S_4(3).2 ,\  S_6(2) ,\  O_8^+(2) ,\  O_8^+(2).2 ,\  O_8^+(2).S_3 ,\  L_3(4).2_1 ,\  L_3(4).2^2 ,\  U_4(3).2_1^2 ,\  U_4(3).2_2^2.\]
\end{remark}

{\footnotesize
	\begin{table}[ht!]
		\begin{center}
			\begin{tabular}{c|cccccc}\hline\hline
				$S$ & $L_2(7)$ & $L_2(8)$ & $L_3(4)$ & $U_3(3)$ & $U_3(4)$ & $U_3(5)$ \\ \hline
				$\Out(S)$ & $2$ & $3$ & $2 \times S_3$ & $2$ & $4$ & $S_3$  \\ \hline
				\cut & $S$ & $S.3$ & \begin{minipage}{2.7cm} $S.2_1$, $S.2_2$, $S.6$, $S.2^2$,  $S.3.2_2$, $S.D_{12}$
				\end{minipage} & $S$, $S.2$ & $S.4$ & $S$, $S.2$  \\ \hline
				not \cut & $S.2$ & $S$ & $S$, $S.2_3$, $S.3$, $S.3.2_3$ & - & $S$, $S.2$ & $S.3$, $S.S_3$ 
				\\  \hline
				& \cite{KKStAndrews} & \cite{KKStAndrews} & \cite{BM42} & \cite{KKStAndrews} & \cite{BM42} & \cite{BM42}  \\ 
				\hline\hline
				$S$ & $U_3(8)$ & $U_4(3)$ & $U_5(2)$ & $U_6(2)$ & $S_4(3) = U_4(2)$   \\ \hline
				$\Out(S)$ & $3 \times S_3$& $D_8$ & $2$ & $S_3$ & $2$  \\ \hline
				\cut & $S.3_1$& \begin{minipage}{2.0cm} $S$, $S.2_1$, $S.2_2$, $S.2_3$, $S.2^2_1$, $S.2^2_2$ \end{minipage} & $S$, $S.2$ & $S$, $S.2$ & $S$, $S.2$   \\ \hline
				not \cut & \begin{minipage}{2.6cm} $S$, $S.2$, $S.3_2$, $S.3_3$, $S.6$, $S.S_3$, $S.3^2$, $S.(S3\times 3)$ \end{minipage}& $S.4$, $S.D_8$ & - & $S.3$, $S.S_3$ & -
				\\  \hline
				& \cite{BM42}& \cite{BM42} & \cite{BM42} & \cite{CM21} & \cite{KKStAndrews}  \\ 
				\hline\hline
				$S$ & $S_6(2)$& $G_2(4)$ & ${}^2F_4(2)'$ & ${}^3D_4(2)$ &$O_8^+(2)$\\ \hline
				$\Out(S)$ & $1$ & $2$ & $2$ & $3$ &$S_3$\\ \hline
				\cut & $S$& $S.2$ & $S.2$ & $S.3$ &$S$, $S.2$, $S.3$, $S.S_3$\\ \hline
				not \cut & - & $S$ & $S$ & $S$ &-\\  \hline
				& \cite{BM42} &  \cite{CM21} & \cite{BM42} & \cite{BM42}&  \cite{BM42}\\ 
				\hline\hline
			\end{tabular} \\[0.5cm]
		\end{center}
		\caption{\label{Table_PQcut}}
	\end{table}
}

\begin{table}[ht!]
	\begin{tabular}{c|c}
		\hline\hline
		\text{Socle} & \text{Almost simple \cut groups} \\\hline
		\text{Alternating} & $S_n$ \; ($n\geqslant 5$), \quad $A_n$ \; ($n\in \{7, 8, 9, 12\}$), \quad $M_{10} = A_6.2_3$, \quad \text{and} \quad $\Aut(A_6) = A_6.2^2$\\\hline
		\text{Sporadic} & $M_{11}$,\ $M_{12}$,\ $M_{22}$,\ $M_{22}.2$,\ $M_{23}$,\ $M_{24}$,\ $Co_1$,\ $Co_2$,\ $Co_3$,\ $HS$,\ $McL$,\ $McL.2$,\ $Th$,\ $M$.\\\hline\hline
	\end{tabular}
	\caption{\label{Table_PQCut2}}
\end{table}

\textbf{Acknowledgment}: We are thankful to the referee for the careful reading and useful suggestions. We thank Anu Rani Jindal for pointing out a mistake in the previous version.

\bibliographystyle{amsalpha}
\bibliography{cut}

\end{document}